\documentclass[a4paper,12pt]{article}
\usepackage{ucs}
\usepackage{amsfonts, amssymb, amsmath, amsthm, amscd}
\usepackage{graphicx}
\usepackage{cite}
\ifx\undefined \pdfgentounicode \else
\input{glyphtounicode} \pdfgentounicode=1
\fi

\author{A.A. Vasil'eva\footnote{Lomonosov Moscow State University, Moscow Center for Fundamental and Applied Mathematics}}
\title{Estimates for the Kolmogorov widths of weighted Sobolev classes with conditions on the $0$‐th and the greatest derivatives\footnote{This research was carried out with the financial support of the Russian Science Foundation (grant no. 22-21-00204).}}
\date{}
\begin{document}

\maketitle

\newenvironment{Biblio}{%
                  \renewcommand{\refname}{\footnotesize REFERENCES}%
                  \begin{thebibliography}{99}%
                  \itemsep -0.2em}%
                  {\end{thebibliography}}

\def\inff{\mathop{\smash\inf\vphantom\sup}}
\renewcommand{\le}{\leqslant}
\renewcommand{\ge}{\geqslant}
\newcommand{\sgn}{\mathrm {sgn}\,}
\newcommand{\inter}{\mathrm {int}\,}
\newcommand{\dist}{\mathrm {dist}}
\newcommand{\supp}{\mathrm {supp}\,}
\newcommand{\R}{\mathbb{R}}
\newcommand{\Z}{\mathbb{Z}}
\newcommand{\N}{\mathbb{N}}
\newcommand{\Q}{\mathbb{Q}}
\theoremstyle{plain}
\newtheorem{Trm}{Theorem}
\newtheorem{trma}{Theorem}

\newtheorem{Def}{Definition}
\newtheorem{Cor}{Corollary}
\newtheorem{Lem}{Lemma}
\newtheorem{Rem}{Remark}
\newtheorem{Sta}{Proposition}
\newtheorem{Sup}{Assumption}
\newtheorem{Supp}{Assumption}
\newtheorem{Not}{Notation}
\newtheorem{Exa}{Example}
\renewcommand{\proofname}{\bf Proof}
\renewcommand{\thetrma}{\Alph{trma}}
\renewcommand{\theSupp}{\Alph{Supp}}

\begin{abstract}
In this article, we obtain the order estimates for the Kolmogorov widths of sets with conditions on the norm in the weighted Sobolev space $W^r_{p_1}$ and in the weighted space $L_{p_0}$.
\end{abstract}

\section{Introduction}

In this paper, we continue the investigation from \cite{vas_inters} about estimating the Kolmogorov widths of the weighted Sobolev classes with conditions on the derivatives of order 0 and the greatest order. Such classes were studied by Oinarov \cite{r_oinarov}, Stepanov and Ushakova \cite{st_ush} (in these papers, sharp two‐sided estimates for the norm of the embedding operator of the weighted Sobolev class on an interval and half‐axis with conditions on 0‐th and the first derivatives were obtained), Triebel \cite{triebel}, Lizorkin and Otelbaev \cite{lo1, lo2, lo3}, Mynbaev and Otelbaev \cite{myn_otel}, Aitenova and Kusainova \cite{ait_kus1, ait_kus2} (in these papers, the problem on estimating the Kolmogorov and the linear widths of weighted Sobolev classes on a domain with conditions on the greatest and 0‐th derivatives in the weighted $L_q$‐space was studied; the conditions on the both derivatives were given in weighted $L_p$‐spaces). In addition, Boykov \cite{boy_1} studied the problem on estimating the Kolmogorov and linear widths of weighted Sobolev classes on a cube with conditions on the derivatives of order from 0 to $r$; the conditions on lower derivatives are given in weighted $L_\infty$‐spaces, and the conditions on higher derivatives are given in weighted $L_p$‐spaces. Boykov and Ryazantsev \cite{boy_ryaz} studied the problem on estimating the Kolmogorov widths of the infinite intersection of weighted Sobolev classes on a cube (the conditions on the derivatives are given in weighted $L_\infty$‐spaces). For details, see \cite{vas_inters}.

First we recall the definition of the Kolmogorov widths. Let $X$ be a normed space, $C\subset X$, and let $n\in \Z_+$. The Kolmogorov widths of the set $C$ in $X$ are defined by
$$
d_n(C, \, X) = \inf _{L\in {\cal L}_n(X)} \sup _{x\in C} \inf _{y\in L}\|x‐y\|,
$$
where ${\cal L}_n(X)$ is the family of linear subspaces in $X$ of dimension at most $n$. For details, see \cite{kniga_pinkusa} and \cite{itogi_nt}.

In \cite{vas_inters}, the problem on estimating the widths $d_n(M, \, L_{q,v}(\Omega))$ was studied, where 
\begin{align}
\label{m_def}
M = \left\{ f:\Omega \rightarrow \R, \;\; \left\| \frac{\nabla ^r f}{g}\right\|_{L_{p_1}(\Omega)}\le 1, \;\; \|wf\|_{L_{p_0}(\Omega)}\le 1\right\},
\end{align}
$$
L_{q,v}(\Omega) =\{f:\Omega \rightarrow \R, \; \|f\|_{L_{q,v}(\Omega)}:= \|vf\|_{L_q(\Omega)}<\infty\}.
$$
In particular, there was considered the following example: $\Omega$ is a John domain, the weights have the form
\begin{align}
\label{gwv}
g(x) = {\rm dist}^{‐\beta}(x, \, \Gamma), \;\; w(x) = {\rm dist}^{‐\sigma}(x, \, \Gamma), \;\; v(x) = {\rm dist}^{‐\lambda}(x, \, \Gamma),
\end{align}
where $\Gamma \subset \partial \Omega$ is an $h$‐set,
\begin{align}
\label{h_theta}
h(t)=t^\theta, \quad 0\le \theta<d.
\end{align}
The definition of a John domain and of an $h$‐set will be given in \S 3. Here we notice that John domains have no zero angles, and that the Sobolev embedding condition for such domains is the same as for a cube \cite{resh1, resh2}. The examples of John domains are domains with Lipschitz boundary and Koch's snowflake. The examples of $h$‐sets are $k$‐dimensional Lipschitz manifolds (with $h(t)=t^k$), Koch's curve, some Cantor‐type sets.

In \cite{vas_inters} order estimates for $d_n(M, \, L_{q,v}(\Omega))$ were obtained, where the weight functions are given by (\ref{gwv}), (\ref{h_theta}), and the additional conditions on the parameters hold:
\begin{align}
\label{rdqpg0beta}
\begin{array}{c}
r +\frac{d}{q}‐\frac{d}{p_1}>0, \;\; r +\frac{d}{p_0} ‐\frac{d}{p_1}>0,
\\
\beta + \sigma ‐ r ‐\frac{d}{p_0} + \frac{d}{p_1} >0, \;\;
\beta + \sigma ‐ r ‐\frac{d‐\theta}{p_0} + \frac{d‐\theta}{p_1} >0.
\end{array}
\end{align}
The case when some of these inequalities does not hold was not studied. Here we consider this case, but instead of $M$ we take the set
\begin{align}
\label{widehat_m}
\widehat M = \left\{ f:\Omega \rightarrow \R, \;\; \left\| \frac{\nabla ^r f}{g}\right\|^{p_1}_{L_{p_1}(\Omega)}+\left\| \frac{f}{g_0}\right\|^{p_1}_{L_{p_1}(\Omega)}\le 1, \;\; \|wf\|_{L_{p_0}(\Omega)}\le 1\right\},
\end{align}
where $g$, $w$, $v$ are given by (\ref{gwv}), (\ref{h_theta}),
\begin{align}
\label{g0_def}
g_0(x) = {\rm dist}^{r‐\beta}(x, \, \Gamma)
\end{align}
(for $M$, there are difficulties with embedding theorems).

In this paper, we obtain the order estimates for $d_n(\widehat{M}, \, L_{q,v}(\Omega))$. The main idea of the proof is the same as in \cite{vas_inters}. The estimating from above is reduced to estimating the sum of the Kolmogorov widths of the intersections of two finite‐dimensional balls; they can be estimated by the widths of balls in $p_0$‐, $p_1$‐, $2$‐ or $q$‐metrics. In order to estimate the widths from below, it is sufficient to estimate the widths of some intersections of two balls. In \cite{vas_inters}, we reduced this problem to considering the multi‐dimensional cube or octahedron; the estimates for the widths of such sets are well‐known \cite{pietsch1, stesin, bib_gluskin}. If (\ref{rdqpg0beta}) does not hold, it is not sufficient to consider the cube or the octahedron; we apply the results of \cite{vas_fin_dim}.

The paper is organized as follows. In \S 2 we obtain the estimates for the widths of the intersection of two balls in some general function spaces. In \S 3 we apply this result for the intersection of two weighted Sobolev classes. Besides the set $\widehat{M}$ given by \eqref{gwv}, \eqref{h_theta}, \eqref{widehat_m}, \eqref{g0_def}, we consider the analogues of other examples from \cite{vas_inters}. In addition, we will study some cases when the widths of $M$ and $\widehat{M}$ have the same orders.

\section{Estimates for the widths of the intersection of two balls in function spaces}

Let $(\Omega, \, \Sigma, \, {\rm mes})$ be a measure space. We say that $A$, $B\subset \Omega$ do not overlap if ${\rm mes}(A\cap
B)=0$. Let $E$, $E_1, \, \dots, \, E_m\subset \Omega$ be measurable sets, $m\in \N\cup \{\infty\}$. We say that $\{E_i\}_{i=1}^m$ is a partition of $E$ if the sets $E_i$ do not overlap and ${\rm mes}\left(\left(\cup _{i=1}^m E_i\right)\bigtriangleup E\right)=0$.

We denote by $\chi_E(\cdot)$ the indicator function of a set $E$.

Let $1< p_0, \, p_1\le \infty$, $1\le q< \infty$. In \cite{vas_inters} the spaces $X_{p_i}(\Omega)$ ($i=0, \, 1$) and $Y_q(\Omega)$ were defined. We recall their properties. For each measurable set $E\subset \Omega$, we define
\begin{itemize}
\item the spaces $X_{p_i}(E)$ with seminorms
$\|\cdot\|_{X_{p_i}(E)}$, $i=0, \, 1$,
\item the Banach space $Y_q(E)$ with norm $\|\cdot\|_{Y_q(E)}$,
\end{itemize}
which satisfy the following conditions:
\begin{enumerate}
\item $X_{p_i}(E)=\{f|_E:\; f\in X_{p_i}(\Omega)\}$, $i=0, \, 1$, $Y_q(E)=\{f|_E:\; f\in
Y_q(\Omega)\}$;
\item if ${\rm mes}\, E=0$, then $\dim \, Y_q(E)=\dim \, X_{p_i}(E)=0$, $i=0, \, 1$;
\item if $E\subset \Omega$, $E_j\subset \Omega$ ($j\in \N$)
are measurable sets, $E=\sqcup _{j\in \N} E_j$, then
$$
\|f\|_{X_{p_i}(E)}=\left\| \bigl\{
\|f|_{E_j}\|_{X_{p_i}(E_j)}\bigr\}_{j\in
\N}\right\|_{l_{p_i}},\quad f\in X_{p_i}(E), \; i=0, \, 1,
$$
$$
\|f\|_{Y_q(E)}=\left\| \bigl\{\|f|_{E_j}\|
_{Y_q(E_j)}\bigr\}_{j\in \N}\right\|_{l_q}, \quad f\in Y_q(E);
$$
\item if $E\in \Sigma$, $f\in Y_q(\Omega)$, then $f\cdot \chi_E\in
Y_q(\Omega)$.
\end{enumerate}

Let ${\cal P}(\Omega)$  be a subspace of dimension $r_0\in \N$ in the space of equivalence classes of measurable functions on $(\Omega, \, \Sigma, \, \mu)$. For each set $E\in
\Sigma$ we denote
$${\cal P}(E)=\{P|_E:\; P\in {\cal P}(\Omega)\}.$$ Let $G\subset
\Omega$ be a measurable set, and let $T$ be a partition of $G$. We write
$$
{\cal S}_{T}(\Omega)=\{f:\Omega\rightarrow \R:\, f|_E\in
{\cal P}(E),\; E\in T, \; f|_{\Omega\backslash G}=0\}.
$$
If $T$ is finite and for each $E\in T$ the inclusion ${\cal
P}(E)\subset Y_q(E)$ holds, then ${\cal S}_{T}(\Omega)\subset
Y_q(\Omega)$ (see property 4).

For each finite partition $T=\{E_j\}_{j=1}^n$ of a set $E$ and for each function $f\in Y_q(\Omega)$ we write
\begin{align}
\label{fpiqt}
\|f\|_{p_i,q,T}=\left(\sum \limits _{j=1}^n
\|f|_{E_j}\|_{Y_q(E_j)} ^{p_i}\right)^{\frac{1}{p_i}}, \quad i=0, \, 1.
\end{align}

We suppose that there are a partition $\{\Omega _{t,j}\}_{t\ge t_0,
\, j\in \hat J_t}$ of $\Omega$ into measurable subsets (here $t_0\in
\Z_+$) and numbers $c\ge 1$, $s_*>0$, $k_*$, $k_{**}\in \N$, $\gamma_*\ge 0$, $\alpha_*\in \R$,
$\mu_*\in \R$, such that the following assumptions hold.

\begin{Supp} \label{supp1}
For each $t\ge t_0$, $j\in \hat J_t$, the inclusion ${\cal P}(\Omega_{t,j}) \subset X_{p_1}(\Omega_{t,j})\cap X_{p_0}(\Omega_{t,j})$ holds. If $s_*+\frac 1q‐\frac{1}{p_1}>0$ or $p_0\ge q$, then $X_{p_1}(\Omega_{t,j})\cap X_{p_0}(\Omega_{t,j})\subset Y_q(\Omega_{t,j})$.
\end{Supp}
\begin{Supp} \label{supp2}
The following estimate holds:
\begin{align}
\label{card_jt} {\rm card}\, \hat J_t\le c\cdot 2^{\gamma_*k_*t},
\quad t\ge t_0.
\end{align}
\end{Supp}
\begin{Supp} \label{supp3}
For each $t\ge t_0$, $j\in \hat J_t$, there is a sequence of partitions $\{T_{t,j,m}\}_{m\in \Z_+}$ of the set $\Omega_{t,j}$ such that
\begin{align}
\label{ttj0} T_{t,j,0} = \{\Omega_{t,j}\}, \quad {\rm card}\,
T_{t,j,m}\le c\cdot 2^m,
\end{align}
and for all $E\in T_{t,j,m}$
\begin{align}
\label{card_e} {\rm card}\, \{E'\in T_{t,j,m\pm 1}:\; {\rm mes}\,
(E\cap E')>0\} \le c.
\end{align}
\end{Supp}
\begin{Supp} \label{supp4}
If $p_0\ge q$, then for each $t\ge t_0$, $j\in \hat J_t$, $m\in \Z_+$, $E\in T_{t,j,m}$, we have
\begin{align}
\label{f_yqe} \|f\|_{Y_q(E)}\le c\cdot 2^{-\alpha_*k_*t}\cdot
2^{m\left(\frac{1}{p_0}-\frac 1q\right)}\|f\|_{X_{p_0}(E)}.
\end{align}
\end{Supp}
\begin{Supp} \label{supp5}
For each $t\ge t_0$, $j\in \hat J_t$, $m\in \Z_+$, $E\in T_{t,j,m}$, there is a linear continuous projection
$P_E:Y_q(\Omega) \rightarrow {\cal S}_{\{E\}}(\Omega)$ with the following properties:
\begin{enumerate}
\item For each function $f\in X_{p_1}(\Omega)\cap X_{p_0}(\Omega)$ we have
\begin{align}
\label{pef} \|P_E(f\cdot \chi_E)\|_{Y_q(E)}\le c\cdot 2^{-\alpha_*k_*t}\cdot
2^{m\left(\frac{1}{p_0}-\frac 1q\right)}\|f\|_{X_{p_0}(E)};
\end{align}
if $m=0$, then, in addition,
\begin{align}
\label{pef1} \|P_E(f\cdot \chi_E)\|_{Y_q(E)}\le c\cdot 2^{\mu_*k_*t}\|f\|_{X_{p_1}(E)}.
\end{align}
\item For each $t\ge t_0$, $j\in \hat J_t$, $f\in Y_q(\Omega)$
\begin{align}
\label{ptmf_to_f} \sum \limits _{E\in T_{t,j,m}} P_E(f\cdot \chi_E) \underset{m\to \infty}{\to} f\cdot \chi _{\Omega_{t,j}}\text{ in the space } Y_q(\Omega).
\end{align}
\item Let $E\in T_{t,j,m}$, $E'\in T_{t,j,m\pm 1}$, ${\rm mes}\, (E\cap E')>0$, $f\in X_{p_1}(\Omega)\cap X_{p_0}(\Omega)$, $\tilde P_Ef$, $\tilde P_{E'}f\in {\cal P}(\Omega)$, $(\tilde P_Ef)|_E=P_E(f\cdot \chi _E)|_{E}$, $(\tilde P_{E'}f)|_{E'}=P_{E'}(f\cdot \chi _{E'})|_{E'}$. Then
\begin{align}
\label{fpef} \|\tilde P_Ef‐\tilde P_{E'}f\|_{Y_q(E\cup E')}\le c \cdot 2^{\mu_*k_*t}\cdot
2^{-m\left(s_*+\frac 1q-\frac{1}{p_1}\right)} \|f\| _{X_{p_1}(E\cup E')}.
\end{align}
\end{enumerate}
\end{Supp}

In addition, we suppose that the following assumption holds.

\begin{Supp} \label{supp6}
For each $t\ge
t_0$, $m\in \Z_+$, there are functions $\varphi_j^{t,m}\in
X_{p_0}(\Omega) \cap X_{p_1}(\Omega)$ $(1\le j\le \nu'_{t,m})$ with pairwise disjoint supports such that
$$
\nu'_{t,m} = \lceil c^{-1}2^{\gamma_*k_{**}t}\cdot 2^m\rceil,
$$
$$ \begin{array}{c} \|\varphi_j^{t,m}\|
_{Y_q(\Omega)} = 1, \quad \|\varphi _j^{t,m}\| _{X_{p_0}(\Omega)}
\le c\cdot 2^{\alpha_*k_{**}t}\cdot2^{m\left(1/q
-1/p_0\right)}, \\ \|\varphi _j^{t,m}\| _{X_{p_1}(\Omega)}
\le c\cdot 2^{-\mu_*k_{**}t}\cdot 2^{m\left(s_*+1/q-1/p_1\right)}.
\end{array}
$$
\end{Supp}

We write $$BX_{p_i}(\Omega) = \{f\in X_{p_i}(\Omega):\;
\|f\|_{X_{p_i}(\Omega)}\le 1\}, \quad i=0, \, 1,$$
$$
M = BX_{p_0}(\Omega) \cap BX_{p_1}(\Omega),
$$
$$
\mathfrak{Z}_0=(p_0, \, p_1, \, q, \, r_0, \, c, \, k_*,\, k_{**},\, s_*, \,
\gamma_*, \, \mu_*, \, \alpha_*).
$$

In \cite{vas_inters} the order estimates for $d_n(M, \, Y_q(\Omega))$ were obtained; it was supposed that Assumptions \ref{supp1}‐‐\ref{supp6} without (\ref{pef1}) and (\ref{ptmf_to_f}) hold, and (\ref{fpef}) was replaced by
$$
\|f-P_E(f\cdot \chi_E)\|_{Y_q(E)}\le c \cdot 2^{\mu_*k_*t}\cdot
2^{-m\left(s_*+\frac 1q-\frac{1}{p_1}\right)} \|f\| _{X_{p_1}(E)};
$$
in addition, it was supposed that
\begin{align}
\label{posit}
\begin{array}{c}
s_* +\frac 1q ‐\frac{1}{p_1}>0, \quad s_*+\frac{1}{p_0}‐\frac{1}{p_1}>0, \\ \mu_*+\alpha_*>0, \quad \mu_*+\alpha_* + \gamma_*/p_0‐\gamma_*/p_1>0. 
\end{array}
\end{align}
Here we obtain the estimates for $d_n(M, \, Y_q(\Omega))$ when at least one of the values in (\ref{posit}) is negative. In addition, we notice the cases when it is possible to obtain the estimates without the condition (\ref{pef1}).

{\bf Auxiliary assertions.} First we write the discretization lemmas.

Let $1\le s\le \infty$. We denote by $l_s^N$ the space $\R^N$ with norm
$$
\|(x_1, \, \dots, \, x_N)\|_{l_s^N} =\left\{ \begin{array}{l} \left(\sum \limits_{j=1}^N |x_j|^s\right)^{1/s}, \quad s<\infty, \\ \max _{1\le j\le N}|x_j|, \quad s=\infty,\end{array}\right.
$$
and by $B_s^N$, the unit ball in $l_s^N$.

Given $t$, $m\in \Z_+$, we denote
\begin{align}
\label{nu_t_m_defin} \nu_{t,m} = \lceil 2\cdot 2^{\gamma_*k_*t}\cdot 2^m\rceil,
\end{align}
\begin{align}
\label{wtm_defin}
W_{t,m} =
2^{\mu_*k_*t}\cdot 2^{‐m(s_*+1/q‐1/p_1)}B_{p_1}^{\nu_{t,m}} \cap 
2^{‐\alpha_*k_*t}\cdot 2^{‐m(1/q‐1/p_0)}B_{p_0}^{\nu_{t,m}}.
\end{align}

The following assertion is obtained in \cite[p. 30]{vas_inters}.
\begin{Lem}
\label{discr_lemma_low}
Let Assumption \ref{supp6} hold. Then for all $t\ge t_0$, $m\ge 0$
\begin{align}
\label{low_est}
d_n(M, \, Y_q(\Omega)) \underset{\mathfrak{Z}_0}{\gtrsim} d_n(W_{t,m}, \, l_q^{\nu_{t,m}}).
\end{align}
\end{Lem}

As in \cite[p. 10]{vas_inters}, we write $$\Omega_t=\cup _{j\in \hat J_t} \Omega _{t,j},$$ define the partitions $T_{t,m}$ and $\hat T_{t,m}$ of the set $\Omega_t$
by
\begin{align}
\label{ttm_def}
T_{t,m} = \{E\in T_{t,j,m}:\; j\in \hat J_t\}, \quad \hat T_{t,m}
= \{E\cap E':\;  E\in T_{t,m}, \; E'\in T_{t,m+1}\},
\end{align}
and obtain
\begin{align}
\label{ttm} {\rm card}\, T_{t,m} \stackrel{(\ref{card_jt}),
(\ref{ttj0})}{\underset{\mathfrak{Z}_0}{\lesssim}} 2^{\gamma
_*k_*t}\cdot 2^m, \quad {\rm card}\, \hat T_{t,m}
\stackrel{(\ref{card_e})}{\underset{\mathfrak{Z}_0}{\lesssim}}
2^{\gamma _*k_*t}\cdot 2^m.
\end{align}
We also define the operator $P_{t,m}:Y_q(\Omega) \rightarrow Y_q(\Omega)$ by
$$
P_{t,m}f = \sum \limits _{j\in \hat J_t} \sum \limits _{E\in
T_{t,j,m}} P_E(f \cdot \chi _E)
$$
and obtain the following estimates \cite[p. 10]{vas_inters}
\begin{align}
\label{rk_ptm} {\rm rk}\, P_{t,m}
\stackrel{(\ref{ttm})}{\underset{\mathfrak{Z}_0}{\lesssim}}
2^{\gamma _*k_*t}\cdot 2^m, \quad {\rm rk}\,
(P_{t,m+1}-P_{t,m})
\stackrel{(\ref{ttm})}{\underset{\mathfrak{Z}_0}{\lesssim}}
2^{\gamma _*k_*t}\cdot 2^m
\end{align}
(here ${\rm rk}$ in the range of an operator),
\begin{align}
\label{ptmp0a}
\|P_{t,m}f\|_{p_0,q,T_{t,m}}
\stackrel{(\ref{pef})}{\underset{\mathfrak{Z}_0}{\lesssim}}
2^{-\alpha_*k_*t}\cdot 2^{m(1/p_0-1/q)}, \quad f\in M,
\end{align}
\begin{align}
\label{ptmptm1p0a}
\|P_{t,m+1}f - P_{t,m}f\| _{p_0,q,\hat
T_{t,m}} \stackrel{(\ref{card_e}),
(\ref{pef})}{\underset{\mathfrak{Z}_0}{\lesssim}}
2^{-\alpha_*k_*t}\cdot 2^{m(1/p_0-1/q)}, \quad f\in M.
\end{align}
Similarly we obtain the estimate
\begin{align}
\label{ptmp1m}
\|P_{t,0}f\|_{p_1,q,T_{t,0}}
\stackrel{(\ref{pef1})}{\underset{\mathfrak{Z}_0}{\lesssim}}
2^{\mu_*k_*t}, \quad f\in M.
\end{align}
We prove that 
\begin{align}
\label{ptmptm1p1m}
\|P_{t,m+1}f - P_{t,m}f\| _{p_1,q,\hat
T_{t,m}} \underset{\mathfrak{Z}_0}{\lesssim}
2^{\mu_*k_*t}\cdot 2^{‐m(s_*+1/q‐1/p_1)}, \quad f\in M.
\end{align}
Indeed, by (\ref{fpiqt}) and (\ref{ttm_def}) we get
$$
\|P_{t,m+1}f - P_{t,m}f\| _{p_1,q,\hat
T_{t,m}} = \left(\sum \limits _{E\in T_{t,m},\,E'\in T_{t,m+1}}\|P_E(f\cdot \chi_E)‐P_{E'}(f\cdot \chi_{E'})\|^{p_1}_{Y_q(E\cap E')}\right)^{1/p_1} \le
$$
$$
\le \left(\sum \limits _{E\in T_{t,m}} \sum \limits _{E'\in T_{t,m+1},\, {\rm mes}(E\cap E')>0}\|\tilde P_Ef‐\tilde P_{E'}f\|^{p_1}_{Y_q(E)}\right)^{1/p_1}\stackrel{(\ref{fpef})}{\underset{\mathfrak{Z}_0}{\lesssim}}
$$
$$
\lesssim 2^{\mu_*k_*t}\cdot 2^{‐m(s_*+1/q‐1/p_1)}\left(\sum \limits _{E\in T_{t,m}} \sum \limits _{E'\in T_{t,m+1},\, {\rm mes}(E\cap E')>0}\|f\|^{p_1}_{X_{p_1}(E\cup E')}\right)^{1/p_1}\stackrel{(\ref{card_e})}{\underset{\mathfrak{Z}_0}{\lesssim}}
$$
$$
\lesssim 2^{\mu_*k_*t}\cdot 2^{‐m(s_*+1/q‐1/p_1)}\|f\|_{X_{p_1}(\Omega_t)} \le 2^{\mu_*k_*t}\cdot 2^{‐m(s_*+1/q‐1/p_1)}.
$$

Let $\tilde \Omega_t = \cup _{l\ge t} \cup _{i\in J_l} \Omega _{l,i}$.

It follows from Assumption \ref{supp1} and (\ref{ptmf_to_f}) that if $s_*+\frac 1q‐\frac{1}{p_1}>0$ or $p_0>q$, then for $f\in X_{p_0}(\Omega)\cap X_{p_1}(\Omega)$ 
$$
\|f‐P_{t,m}f\|_{Y_q(\Omega_t)} \underset{m\to \infty}{\to}0.
$$
Hence, for each $\hat t\ge t_0$
$$
f = \sum \limits _{t=t_0}^{\hat t} P_{t,0}f + \sum \limits _{t=t_0}^{\hat t} \sum \limits _{m=0}^\infty (P_{t,m+1}f ‐ P_{t,m}f) +f\cdot \chi _{\tilde \Omega _{\hat t+1}}.
$$

In \cite[Proposition 3]{vas_inters} it was proved that for $l\in \Z_+$
$$
d_l((P_{t,m+1}‐P_{t,m})M, \, Y_q(\Omega)) \underset{\mathfrak{Z}_0}{\lesssim} d_l(W_{t+\tau_0,m}, \, l_q^{\nu_{t+\tau_0,m}}),
$$
$$
d_l(P_{t,0}M, \, Y_q(\Omega)) \underset{\mathfrak{Z}_0}{\lesssim} 2^{‐\alpha_*k_*t}d_l(B_{p_0}^{\nu_{t+\tau_0,m}}, \, l_q^{\nu_{t+\tau_0,m}}),
$$
where $\tau_0\in \Z_+$ depends only on $\mathfrak{Z}_0$. Here we applied (\ref{nu_t_m_defin}), (\ref{rk_ptm}), (\ref{ptmp0a}), (\ref{ptmptm1p0a}), (\ref{ptmptm1p1m}). Employing (\ref{ptmp1m}) together with (\ref{ptmp0a}), we similarly obtain that
$$
d_l(P_{t,0}M, \, Y_q(\Omega)) \underset{\mathfrak{Z}_0}{\lesssim} d_l(W_{t+\tau_0,0}, \, l_q^{\nu_{t+\tau_0,0}}).
$$

This yields the following assertion.

\begin{Lem}
\label{discr_lemma}
Let Assumptions \ref{supp1}‐‐\ref{supp5} hold, let $n\in \Z_+$, $\hat t(n)\ge t_0$, $k_{t,m}\in \Z_+$, $C\in \N$, $\sum \limits _{t=t_0}^{\hat t(n)}\sum \limits_{m=0}^\infty k_{t,m} \le Cn$. Then there is $C_1=C_1(\mathfrak{Z}_0)\in \N$ such that
\begin{align}
\label{up_est}
d_{C_1Cn}(M, \, Y_q(\Omega)) \underset{\mathfrak{Z}_0}{\lesssim} \sum \limits _{t=0}^{\hat t(n)}\sum \limits_{m=0}^\infty d_{k_{t,m}}(W_{t,m}, \, l_q^{\nu_{t,m}}) + \sup _{f\in M} \|f\| _{Y_q(\tilde \Omega _{\hat t(n)+1})}.
\end{align}
In addition,
\begin{align}
\label{embed_est_width} \sup _{f\in M} \|f\|_{Y_q (\tilde \Omega _t)} \underset{\mathfrak{Z}_0}{\lesssim} \sum \limits _{l=t}^\infty\sum \limits_{m=0}^\infty d_0(W_{l,m}, \, l_q^{\nu_{l,m}}).
\end{align}
\end{Lem}

\begin{Sta}
\label{discr_up1} Let Assumptions \ref{supp1}‐‐\ref{supp5} hold without (\ref{pef1}), let $n\in \Z_+$, $\hat t(n)\ge t_0$, $k_{t,m}\in \Z_+$, $s_t\in \Z_+$, $C\in \N$, $\sum \limits _{t=t_0}^{\hat t(n)}\sum \limits_{m=0}^\infty k_{t,m} + \sum \limits _{t=t_0}^{\hat t(n)} s_{t}\le Cn$. Then there is $C_1=C_1(\mathfrak{Z}_0)$ such that
\begin{align}
\label{up_est0}
\begin{array}{c}
d_{C_1Cn}(M, \, Y_q(\Omega)) \underset{\mathfrak{Z}_0}{\lesssim} \sum \limits _{t=0}^{\hat t(n)}\sum \limits_{m=0}^\infty d_{k_{t,m}}(W_{t,m}, \, l_q^{\nu_{t,m}})+\\ +\sum \limits _{t=0}^{\hat t(n)} d_{s_t}(2^{‐\alpha_*k_*t}B_{p_0}^{\nu_{t,0}}, \, l_q^{\nu_{t,0}}) + \sup _{f\in M} \|f\| _{Y_q(\tilde \Omega _{\hat t(n)+1})}.
\end{array}
\end{align}

\end{Sta}

In fact, the analogue of Proposition \ref{discr_up1} was proved in \cite{vas_inters}.

In \cite{vas_fin_dim}, estimates for the Kolmogorov widths of $B_{p_0}^N\cap \nu B_{p_1}^N$ for $p_0>p_1$ were obtained. Rewrite this result for the widths of $\nu_0B_{p_0}^N\cap \nu_1 B_{p_1}^N$.

In what follows, we define the numbers $\lambda$ and $\tilde \lambda$ by 
\begin{align}
\label{lam_til_lam}
\frac 1q=\frac{1‐\lambda}{p_1} +\frac{\lambda}{p_0}, \quad \frac 12=\frac{1‐\tilde\lambda}{p_1} +\frac{\tilde\lambda}{p_0}.
\end{align}
Notice that 
\begin{align}
\label{s_1_lam}
(1‐\lambda)\left(s_*+\frac 1q‐\frac{1}{p_1}\right) +\lambda\left(\frac 1q‐\frac{1}{p_0}\right) = (1‐\lambda)s_*. 
\end{align}

If $p_0\le q\le p_1$ or $p_1\le q\le p_0$, then
\begin{align}
\label{wtm_lq_emb} \nu_0B_{p_0}^N\cap \nu_1 B_{p_1}^N \subset \nu_0^\lambda \nu_1^{1‐\lambda}B^N_q;
\end{align}
if $p_0\le 2\le p_1$ or $p_1\le 2\le p_0$, then
$$
\nu_0B_{p_0}^N\cap \nu_1 B_{p_1}^N \subset \nu_0^{\tilde \lambda} \nu_1^{1‐\tilde\lambda}B^N_2.
$$
It follows from H\"{o}lder's inequality or from Galeev's result \cite[Theorem
2]{galeev1}, \cite[Theorem 1]{galeev4}.

We formulate the corollary of the main result of \cite{vas_fin_dim} (see Theorem 1 and remarks before and after it).
\begin{trma}
\label{dn_inters}
Let $1\le p_1<p_0\le \infty$, $1\le q<\infty$, $N\in \N$, $n\in \Z_+$, $n\le N/2$. We define the numbers $\lambda$ and $\tilde \lambda$ by (\ref{lam_til_lam}).

If $\nu_1/\nu_0 \le 1$, then 
\begin{align}
\label{sved_k_p1}
d_n(\nu_0B_{p_0}^N\cap \nu_1 B_{p_1}^N, \, l_q^N)=d_n(\nu_1 B_{p_1}^N, \, l_q^N).
\end{align}

If $\nu_1/\nu_0\ge N^{1/p_1‐1/p_0}$, then 
\begin{align}
\label{sved_k_p0}
d_n(\nu_0B_{p_0}^N\cap \nu_1 B_{p_1}^N, \, l_q^N)=d_n(\nu_0 B_{p_0}^N, \, l_q^N).
\end{align}

Let $1<\nu_1/\nu_0< N^{1/p_1‐1/p_0}$. We set
$\kappa = (\nu_1/\nu_0)^{\frac{\frac 12‐\frac 1q}{\frac{1}{p_1}‐\frac{1}{p_0}}}$. Then the following estimates hold.
\begin{enumerate}
\item Let $1\le p_1< p_0\le q\le 2$ or $1\le p_1< p_0\le 2< q$. Then
$$
d_n(\nu_0B_{p_0}^N\cap \nu_1 B_{p_1}^N, \, l_q^N)
\underset{p_0,p_1,q}{\asymp} d_n(\nu_0 B_{p_0}^N, \, l_q^N).
$$

\item Let $1\le p_1<2<p_0<q<\infty$. Then
$$
d_n(\nu_0B_{p_0}^N\cap \nu_1 B_{p_1}^N, \, l_q^N)
\underset{p_0,p_1,q}{\asymp} \left\{ \begin{array}{l} d_n(\nu_0 B_{p_0}^N, \, l_q^N) \quad
\text{if} \quad n^{\frac 12} \cdot N^{‐\frac 1q}\le \kappa,\\
d_n(\nu_1^{1‐\tilde \lambda}\nu_0^{\tilde \lambda} B_2^N, \, l_q^N) \quad
\text{if} \quad n^{\frac 12} \cdot N^{‐\frac 1q}\ge \kappa.
\end{array}\right.
$$
\item Let $2\le p_1<p_0<q<\infty$. Then
$$
d_n(\nu_0B_{p_0}^N\cap \nu_1 B_{p_1}^N, \, l_q^N)
\underset{p_0,p_1,q}{\asymp} \left\{ \begin{array}{l} d_n(\nu_0 B_{p_0}^N, \, l_q^N) \quad
\text{if} \quad n^{\frac 12} \cdot N^{‐\frac 1q}\le \kappa,\\
d_n(\nu_1 B_{p_1}^N, \, l_q^N) \quad \text{if} \quad
n^{\frac 12} \cdot N^{‐\frac 1q}\ge \kappa.
\end{array}\right.
$$
\item Let $2\le p_1\le q\le p_0$. Then
$$
d_n(\nu_0B_{p_0}^N\cap \nu_1 B_{p_1}^N, \, l_q^N)
\underset{p_0,p_1,q}{\asymp} \left\{ \begin{array}{l}
d_n(\nu_1^{1‐\lambda}\nu_0^{\lambda}B_q^N, \, l_q^N) \quad
\text{if} \quad n^{\frac 12} \cdot N^{‐\frac 1q}\le \kappa,\\
d_n(\nu_1 B_{p_1}^N, \, l_q^N) \quad \text{if} \quad
n^{\frac 12} \cdot N^{‐\frac 1q}\ge \kappa.
\end{array}\right.
$$
\item Let $1\le p_1<2<q\le p_0$. Then
$$
d_n(\nu_0B_{p_0}^N\cap \nu_1 B_{p_1}^N, \, l_q^N)
\underset{p_0,p_1,q}{\asymp} \left\{ \begin{array}{l}
d_n(\nu_1^{1‐\lambda}\nu_0^{\lambda}B_q^N, \, l_q^N) \quad
\text{if} \quad n^{\frac 12} \cdot N^{‐\frac 1q}\le \kappa,\\
 d_n(\nu_1^{1‐\tilde \lambda} \nu_0^{\tilde \lambda}B_2^N, \, l_q^N) \quad
\text{if} \quad n^{\frac 12} \cdot N^{‐\frac 1q}\ge \kappa.
\end{array}\right.
$$
\item Let $q\le 2$, $1\le p_1<q< p_0$. Then
$$
d_n(\nu_0B_{p_0}^N\cap \nu_1 B_{p_1}^N, \, l_q^N)
\underset{p_0,p_1,q}{\asymp} d_n(\nu_1^{1‐\lambda}\nu_0^{\lambda}B_q^N, \, l_q^N).
$$
\item Let $1\le q\le p_1<p_0\le \infty$. Then
$$
d_n(\nu_0B_{p_0}^N\cap \nu_1 B_{p_1}^N, \, l_q^N)
\underset{p_0,p_1,q}{\asymp} d_n(\nu_1B_{p_1}^N, \, l_q^N).
$$
\end{enumerate}
\end{trma}
For $p_0<p_1$, the estimates can be rewrited by rearranging the indices 0 and 1.

The estimates of the widths of $B_p^N$ in $l_q^N$ were obtained in \cite{pietsch1, stesin, bib_kashin, bib_gluskin, garn_glus, k_p_s, stech_poper, kashin_oct} (for details, see \cite{kniga_pinkusa, itogi_nt}). We formulate these results for the cases we need below.

\begin{trma}
\label{glus} {\rm \cite{bib_gluskin}} Let $1\le p\le q<\infty$,
$0\le n\le N/2$.
\begin{enumerate}
\item Let $1\le q\le 2$. Then $d_n(B_p^N, \, l_q^N) \underset{p,q}{\asymp}
1$.

\item Let $2<q<\infty$, $\lambda_{pq} =\min \left\{1, \,
\frac{1/p-1/q}{1/2-1/q}\right\}$. Then
$$
d_n(B_p^N, \, l_q^N) \underset{p,q}{\asymp} \min \{1, \,
n^{-1/2}N^{1/q}\} ^{\lambda_{pq}}.
$$
\end{enumerate}
\end{trma}

\begin{trma}
\label{p_s} {\rm \cite{pietsch1, stesin}} Let $1\le q\le p\le
\infty$, $0\le n\le N$. Then
$$
d_n(B_p^N, \, l_q^N) = (N-n)^{1/q-1/p}.
$$
\end{trma}

{\bf The embedding theorems for $X_{p_1}(\Omega) \cap X_{p_0}(\Omega)$ into $Y_q(\Omega)$.}

Notice that if $p_0\le q$, then the condition $s_*+\frac 1q‐\frac{1}{p_1}>0$ is necessary for the compact embedding. Indeed, if $s_*+\frac 1q‐\frac{1}{p_1}\le 0$, then for each $m\in \Z_+$ we have
$2^{‐m(s_*+1/q‐1/p_1)}\ge 1$, $2^{‐m(1/q‐1/p_0)}\ge 1$; hence $$d_n(W_{t_0,m}, \, l_q^{\nu_{t_0,m}}) \stackrel{(\ref{wtm_defin})}{\underset{\mathfrak{Z}_0, \, t_0}{\gtrsim}} d_n(B_1^{\nu_{t_0,m}}, \, l_q^{\nu_{t_0,m}}) \underset{q}{\gtrsim} 1$$ for sufficiently large $m$ (see Theorem \ref{glus} and (\ref{nu_t_m_defin})). Hence by Lemma \ref{discr_lemma_low}, we have $d_n(M, \, Y_q(\Omega)) \underset{\mathfrak{Z}_0}{\gtrsim} 1$. In what follows we assume that $s_*+\frac 1q‐\frac{1}{p_1}>0$ or $p_0>q$.

We denote $x_+=\max\{x, \, 0\}$ for $x\in \R$.

Let the number $\nu_*$ be defined as follows.
\begin{enumerate}
\item Let one of the following conditions hold: a) $\mu_*+\alpha_*\le 0$, $\mu_*+\alpha_*+\gamma_*/p_0 ‐ \gamma_*/p_1\le 0$, $s_*+1/q‐1/p_1>0$; b) $p_0<p_1\le q$, $\mu_*+\alpha_*\le 0$, $s_*+1/q‐1/p_1>0$, c) $p_0>p_1\ge q$, $\mu_*+\alpha_*+\gamma_*/p_0‐\gamma_*/p_1\le 0$. Then 
\begin{align}
\label{nu1}
\nu_*=‐\mu_*‐\left(\frac{\gamma_*}{q} ‐ \frac{\gamma_*}{p_1}\right)_+.
\end{align}

\item Let one of the following conditions hold: a) $p_1<q < p_0$, $\mu_*+\alpha_*+\gamma_*/p_0‐ \gamma_*/p_1 \le 0 \le\mu_*+\alpha_*$, b) $p_0<q<p_1$, $\mu_*+\alpha_* \le 0 \le \mu_* + \alpha_* +\gamma_*/p_0 ‐ \gamma_*/p_1$. Then 
\begin{align}
\label{nu2}
\nu_* = \frac{\alpha_*(1/p_1‐1/q)+\mu_*(1/p_0‐1/q)}{1/p_1‐1/p_0}.
\end{align}
\item Let $p_0\ge q$, $ \mu_* + \alpha_* +\gamma_*/p_0 ‐ \gamma_*/p_1 \ge 0$. Then 
\begin{align}
\label{nu3}
\nu_* = \alpha_* +\frac{\gamma_*}{p_0} ‐ \frac{\gamma_*}{q}.
\end{align}

\item Let one of the following conditions hold: a) $\mu_* + \alpha_* +\gamma_*/p_0 ‐ \gamma_*/p_1 <0\le \mu_*+\alpha_*$, $p_1\le p_0\le q$, $s_*+\frac 1q‐\frac{1}{p_1}>0$, b) $\mu_*+\alpha_*\le 0$, $p_1<q< p_0$, $s_*+\frac 1q ‐ \frac{1}{p_1}<0$. Then
\begin{align}
\label{nu4}
\nu_* = \frac{\alpha_*(s_*+1/q‐1/p_1)+\mu_*(1/q‐1/p_0)}{s_*+1/p_0‐1/p_1}.
\end{align}
\end{enumerate}

\begin{Sta}
\label{p0_ge_q} 
Let $\nu_*>0$ be defined by (\ref{nu1})‐‐(\ref{nu4}). Then for each function $f\in M$ and for each $t\ge t_0$
\begin{align}
\label{emb_nu} \|f\|_{Y_q(\tilde \Omega_t)} \underset{\mathfrak{Z}_0}{\lesssim} 2^{‐\nu_*k_*t}.
\end{align}
\end{Sta}
\begin{proof}
Notice that $d_0(B_p^N, \, l_q^N) = N^{(1/q‐1/p)_+}$.

In case 1, we have $s_*+\frac 1q‐\frac{1}{p_1}>0$; hence
$$
\|f\|_{Y_q(\tilde \Omega_t)} \stackrel{(\ref{embed_est_width})}{\underset{\mathfrak{Z}_0}{\lesssim}} \sum \limits _{l\ge t} \sum \limits _{m\ge 0} d_0(W_{l,m}, \, l_q^{\nu_{l,m}}) \stackrel{(\ref{wtm_defin})}{\le} \sum \limits _{l\ge t} \sum \limits _{m\ge 0} 2^{\mu_*k_*l}\cdot 2^{‐m(s_*+1/q‐1/p_1)}d_0(B_{p_1}^{\nu_{l,m}}, \, l_q^{\nu_{l,m}}) \stackrel{(\ref{nu_t_m_defin})}{\underset{\mathfrak{Z}_0}{\lesssim}}
$$
$$
\lesssim \sum \limits _{l\ge t} \sum \limits _{m\ge 0} 2^{\mu_*k_*l}\cdot 2^{‐m(s_*+1/q‐1/p_1)}\cdot 2^{(\gamma_*k_*l+m)(1/q‐1/p_1)_+} \underset{\mathfrak{Z}_0}{\lesssim} 2^{(\mu_*+(\gamma_*/q‐\gamma_*/p_1)_+)k_*t}.
$$

Consider case 2. We define $\lambda\in (0, \, 1)$  by (\ref{lam_til_lam}). Then
$$
\|f\|_{Y_q(\tilde \Omega_t)} \stackrel{(\ref{embed_est_width})}{\underset{\mathfrak{Z}_0}{\lesssim}} \sum \limits _{l\ge t} \sum \limits _{m\ge 0} d_0(W_{l,m}, \, l_q^{\nu_{l,m}}) \stackrel{(\ref{s_1_lam}), (\ref{wtm_lq_emb})}{\le}$$$$\le \sum \limits _{l\ge t} \sum \limits _{m\ge 0} 2^{(\mu_*(1‐\lambda)‐\alpha_*\lambda)k_*l}\cdot 2^{‐ms_*(1‐\lambda)} \underset{\mathfrak{Z}_0}{\lesssim} 2^{(\mu_*(1‐\lambda)‐\alpha_*\lambda)k_*t}.
$$

In case 3, the assertion follows from (\ref{f_yqe}) and H\"{o}lder's inequality  (see \cite[Proposition 4]{vas_inters}).

Consider case 4. We define the number $m_l$ by 
\begin{align}
\label{ml_def_emb}
2^{‐\alpha_*k_*l}\cdot 2^{‐m_l(1/q‐1/p_0)}=2^{\mu_*k_*l}\cdot 2^{‐m_l(s_*+1/q‐1/p_1)}. 
\end{align}
Let a) hold. If $p_0=q$, then as in case 3 we get $\|f\|_{Y_q(\tilde \Omega_t)} \underset{\mathfrak{Z}_0}{\lesssim} 2^{‐\alpha_*k_*t}$. It is equivalent to (\ref{emb_nu}), where $\nu_*$ is given by (\ref{nu4}). Let $p_0<q$. Notice that $s_*+\frac{1}{p_0}‐\frac{1}{p_1}> s_*+\frac{1}{q}‐\frac{1}{p_1}>0$. Since $\mu_*+\alpha_*\ge 0$, we have $m_l\ge 0$ for $l\ge 0$ (see (\ref{ml_def_emb})). Hence,
$$
\|f\|_{Y_q(\tilde \Omega_t)} \stackrel{(\ref{embed_est_width})}{\underset{\mathfrak{Z}_0}{\lesssim}} \sum \limits _{l\ge t} \sum \limits _{m\ge 0} d_0(W_{l,m}, \, l_q^{\nu_{l,m}})\stackrel{(\ref{wtm_defin})}{\underset{\mathfrak{Z}_0}{\lesssim}}
$$
$$
\lesssim \sum \limits _{l\ge t} \sum \limits _{0\le m\le m_l} 2^{‐\alpha_*k_*l}\cdot 2^{‐m(1/q‐1/p_0)} + \sum \limits _{l\ge t} \sum \limits _{m> m_l} 2^{\mu_*k_*l}\cdot 2^{‐m(s_*+1/q‐1/p_1)} \stackrel{(\ref{ml_def_emb})}{\underset{\mathfrak{Z}_0}{\lesssim}}
$$
$$
\lesssim \sum \limits _{l\ge t} 2^{\mu_*k_*l}\cdot 2^{‐m_l(s_*+1/q‐1/p_1)} \stackrel{(\ref{nu4}), (\ref{ml_def_emb})}{\underset{\mathfrak{Z}_0}{\lesssim}} 2^{‐\nu_*k_*t}.
$$
Let b) hold. Since $\mu_*+\alpha_*\le 0$, $s_*+\frac{1}{p_0}‐\frac{1}{p_1}< s_*+\frac 1q‐\frac{1}{p_1}<0$, we have $m_l\ge 0$ for $l\ge 0$ (see (\ref{ml_def_emb})). We define $\lambda\in (0, \, 1)$ by (\ref{lam_til_lam}) and get
$$
\|f\|_{Y_q(\tilde \Omega_t)} \stackrel{(\ref{embed_est_width})}{\underset{\mathfrak{Z}_0}{\lesssim}} \sum \limits _{l\ge t} \sum \limits _{m\ge 0} d_0(W_{l,m}, \, l_q^{\nu_{l,m}}) \stackrel{(\ref{wtm_defin}), (\ref{s_1_lam}), (\ref{wtm_lq_emb})}{\le}
$$
$$
\le \sum \limits _{l\ge t} \sum \limits _{0\le m\le m_l} 2^{\mu_*k_*l}\cdot 2^{‐m(s_*+1/q‐1/p_1)} + \sum \limits _{l\ge t} \sum \limits _{m> m_l} 2^{(\mu_*(1‐\lambda)‐\alpha_*\lambda)k_*l}\cdot 2^{‐s_*(1‐\lambda)m}\stackrel{(\ref{ml_def_emb})}{\underset{\mathfrak{Z}_0}{\lesssim}}
$$
$$
\lesssim \sum \limits _{l\ge t} 2^{\mu_*k_*l}\cdot 2^{‐m_l(s_*+1/q‐1/p_1)} \stackrel{(\ref{nu4}), (\ref{ml_def_emb})}{\underset{\mathfrak{Z}_0}{\lesssim}} 2^{‐\nu_*k_*t}.
$$
This completes the proof.
\end{proof}

{\bf Estimates for the widths.}

We denote
\begin{align}
\label{tilde_theta_def}
\tilde \theta = s_*\frac{\alpha_*+\gamma_*/p_0‐\gamma_*/q}{\mu_*+\alpha_*+\gamma_*(s_*+1/p_0‐1/p_1)},
\end{align}
\begin{align}
\label{hat_theta_def}
\hat \theta = \frac{\mu_*(1/q‐1/p_0)+\alpha_*(s_*+1/q‐1/p_1)}{\mu_*+\alpha_*+\gamma_*(s_*+1/p_0‐1/p_1)},
\end{align}
\begin{align}
\label{hat_nu_def}
\hat\nu = \frac 12\cdot \frac{\alpha_*(1/p_1‐1/q)+\mu_*(1/p_0‐1/q)}{(\mu_*+\alpha_*)(1/2‐1/q)+\gamma_*(1/p_1‐1/p_0)/q},
\end{align}
\begin{align}
\label{tilde_nu_def}
\tilde \nu = \frac{\mu_*(1/p_0‐1/2)+\alpha_*(1/p_1‐1/2)}{\gamma_*(1/p_1‐1/p_0)}+\frac 12‐\frac 1q.
\end{align}

We define the numbers $j_0$ and $\theta_j$ $(1\le j\le j_0)$ as follows.

First we consider the case $s_*+\frac{1}{\max\{p_0, \, q\}} ‐\frac{1}{p_1}>0$ or $\min \{p_0, \, p_1\}\ge q$.

\begin{Not}
\label{not2} Let $s_*+\frac{1}{\max\{p_0, \, q\}} ‐\frac{1}{p_1}>0$, $\mu_*+\alpha_*\le 0$, $\mu_* +\alpha_* +\gamma_*/p_0 ‐\gamma_*/p_1\le 0$; if $\gamma_*=0$, we suppose that $\mu_*<0$ and set $‐\mu_*/\gamma_*=+\infty$.
\begin{itemize}
\item If $p_1\ge q$ or $q\le 2$, we set $j_0=2$, $\theta_1=s_*‐\left(\frac{1}{p_1}‐\frac 1q\right)_+$, $\theta_2 = ‐\frac{\mu_*}{\gamma_*} ‐ \left(\frac 1q ‐ \frac{1}{p_1}\right)_+$. 

\item If $p_1<q$, $q>2$, we set $j_0=4$, $\theta_1 = s_*+\min \left\{ 0, \, \frac 12‐\frac{1}{p_1}\right\}$, $\theta_2 =\frac q2 \left(s_*+\frac 1q‐\frac{1}{p_1}\right)$, $\theta_3 = ‐\frac{\mu_*}{\gamma_*} +\min\left\{\frac 12‐\frac 1q, \, \frac{1}{p_1}‐\frac 1q\right\}$, $\theta_4 = ‐\frac{q\mu_*}{2\gamma_*}$.
\end{itemize}
\end{Not}

\begin{Not}
\label{not3} Let $p_0\ge q$, $p_1\ge q$; if $\mu_*+\alpha_*+\gamma_*/p_0‐\gamma_*/p_1\le 0$, $\gamma_*=0$, we suppose that $\mu_*<0$ and set $‐\mu_*/\gamma_*=+\infty$. Then $j_0=2$, $\theta_1=s_*$, $$\theta_2=\left\{\begin{array}{l}\tilde \theta \quad \text{if }\mu_*+\alpha_*+\gamma_*/p_0‐\gamma_*/p_1>0, \\ ‐\frac{\mu_*}{\gamma_*} ‐\frac 1q +\frac{1}{p_1} \quad \text{if }\mu_*+\alpha_*+\gamma_*/p_0‐\gamma_*/p_1\le 0.\end{array}\right.$$
\end{Not}

\begin{Not}
\label{not4} Let $s_*+\frac{1}{\max\{p_0, \, q\}} ‐\frac{1}{p_1}>0$, $\mu_*+\alpha_*< 0$, $\mu_*+\alpha_*+\gamma_*/p_0‐\gamma_*/p_1>0$.
\begin{enumerate}
\item Let $p_0<p_1< q$. 
\begin{itemize}
\item If $q\le 2$, then $j_0=2$, $\theta_1= s_*+\frac 1q ‐\frac{1}{p_1}$, $\theta_2 = ‐\frac{\mu_*}{\gamma_*}$.

\item If $q>2$, $p_1\le 2$, then $j_0=4$, $\theta_1 = s_*+\frac 12‐\frac{1}{p_1}$, $\theta_2 = \frac q2\left(s_*+\frac 1q ‐ \frac{1}{p_1}\right)$, $\theta_3 = ‐\frac{\mu_*}{\gamma_*} +\frac 12 ‐ \frac 1q$, $\theta_4 = ‐\frac{q\mu_*}{2\gamma_*}$.

\item If $q>2$, $p_0\ge 2$, then $j_0=5$, $\theta_1=s_*$, $\theta_2 = \frac q2\left(s_*+\frac 1q ‐ \frac{1}{p_1}\right)$, $\theta_3=\tilde \theta$, $\theta_4 = ‐\frac{q\mu_*}{2\gamma_*}$, $\theta_5 = \hat \nu$.

\item If $q>2$, $p_0<2<p_1$, then $j_0=6$, $\theta_1=s_*$, $\theta_2 = \frac q2\left(s_*+\frac 1q ‐ \frac{1}{p_1}\right)$, $\theta_3=\tilde \theta$, $\theta_4 = ‐\frac{q\mu_*}{2\gamma_*}$, $\theta_5 = \hat \nu$, $\theta_6 = \tilde \nu$.

\end{itemize}

\item Let $p_0<q<p_1$.
\begin{itemize}
\item If $q\le 2$, then $j_0=3$, $\theta_1=s_*$, $\theta_2 = \tilde\theta$, $\theta_3 = \frac{\alpha_*(1/q‐1/p_1)+\mu_*(1/q‐1/p_0)}{\gamma_*(1/p_0‐1/p_1)}$.

\item If $q>2$, $p_0\ge 2$, then $j_0=3$, $\theta_1=s_*$, $\theta_2 = \tilde\theta$, $\theta_3=\hat \nu$.

\item If $q>2$, $p_0<2$, then $j_0=4$, $\theta_1=s_*$, $\theta_2 = \tilde\theta$, $\theta_3=\hat \nu$, $\theta_4=\tilde \nu$.
\end{itemize}

\end{enumerate}
\end{Not}

\begin{Not}
\label{not5} Let $s_*+\frac{1}{\max\{p_0, \, q\}} ‐\frac{1}{p_1}> 0$, $\mu_*+\alpha_*> 0$, $\mu_*+\alpha_*+\gamma_*/p_0‐\gamma_*/p_1< 0$.

\begin{enumerate}
\item Let $p_1<p_0<q$.
\begin{itemize}
\item If $q\le 2$, then $j_0=2$, $\theta_1= s_*+\frac 1q ‐\frac{1}{p_1}$, $\theta_2 = \hat \theta$.

\item If $q>2$, $p_0\le 2$, then $j_0=4$, $\theta_1 = s_*+\frac 12‐\frac{1}{p_1}$, $\theta_2 = \frac q2\left(s_*+\frac 1q ‐ \frac{1}{p_1}\right)$, $\theta_3 = \hat \theta+\frac 12 ‐ \frac 1q$, $\theta_4 = \frac{q\hat \theta}{2}$.

\item If $q>2$, $p_1\ge 2$, then $j_0=5$, $\theta_1=s_*$, $\theta_2 = \frac q2\left(s_*+\frac 1q ‐ \frac{1}{p_1}\right)$, $\theta_3=\frac{q\hat \theta}{2}$, $\theta_4 = ‐\frac{\mu_*}{\gamma_*} ‐\frac 1q+\frac{1}{p_1}$, $\theta_5 = \hat \nu$.

\item If $q>2$, $p_1<2<p_0$, then $j_0=6$, $\theta_1=s_*+\frac 12 ‐\frac{1}{p_1}$, $\theta_2 = \frac q2\left(s_*+\frac 1q ‐ \frac{1}{p_1}\right)$, $\theta_3=\hat \theta +\frac 12 ‐\frac 1q$, $\theta_4=\frac{q\hat \theta}{2}$, $\theta_5 = \hat \nu$, $\theta_6 = \tilde \nu$.

\end{itemize}
\item Let $p_1<q<p_0$.
\begin{itemize}
\item If $q\le 2$, then $j_0=3$, $\theta_1=s_*+\frac 1q‐\frac{1}{p_1}$, $\theta_2 = \hat\theta$, $\theta_3 = \frac{\alpha_*(1/p_1‐1/q)+\mu_*(1/p_0‐1/q)}{\gamma_*(1/p_1‐1/p_0)}$.

\item If $q>2$, $p_1\ge 2$, then $j_0=5$, $\theta_1=s_*$, $\theta_2 = \frac q2\left(s_*+\frac 1q‐\frac{1}{p_1}\right)$, $\theta_3 = \frac{q\hat\theta}{2}$, $\theta_4 = ‐\frac{\mu_*}{\gamma_*} ‐\frac 1q +\frac{1}{p_1}$, $\theta_5=\hat \nu$.

\item If $q>2$, $p_1<2$, then $j_0=6$, $\theta_1=s_* +\frac 12‐\frac{1}{p_1}$, $\theta_2 = \frac q2\left(s_*+\frac 1q‐\frac{1}{p_1}\right)$, $\theta_3=\hat \theta+\frac 12‐\frac 1q$, $\theta_4 = \frac{q\hat\theta}{2}$,  $\theta_5=\hat \nu$, $\theta_6=\tilde \nu$.
\end{itemize}

\end{enumerate}
\end{Not}

Now we consider the case $p_0>q>p_1$, $s_*+\frac{1}{p_0} ‐ \frac{1}{p_1}< 0$. We set
\begin{align}
\label{hat_sigma_def}
\hat \sigma = s_*\cdot \frac{\frac 1q ‐ \frac{1}{p_0}}{‐s_*‐\frac{1}{p_0}+\frac{1}{p_1}+\frac{2s_*}{q}}.
\end{align}

\begin{Not}
\label{not6} Let $p_0>q>p_1$, $s_*+\frac{1}{p_0} ‐ \frac{1}{p_1}< 0$. 
\begin{enumerate}
\item Let $\mu_*+\alpha_*+\gamma_*/p_0‐\gamma_*/p_1\ge 0$.
\begin{itemize}
\item If $q\le 2$, then $j_0=2$, $\theta_1 = s_*\frac{1/q‐1/p_0}{1/p_1‐1/p_0}$, $\theta_2=\tilde \theta$.

\item If $q>2$, $p_1\le 2$, then $j_0=2$, $\theta_1 = \hat \sigma$, $\theta_2=\tilde \theta$.

\item If $q>2$, $p_1> 2$, then $j_0=3$, $\theta_1=s_*$, $\theta_2 = \hat \sigma$, $\theta_3=\tilde \theta$.
\end{itemize}

\item Let $\mu_*+\alpha_*+\gamma_*/p_0‐\gamma_*/p_1< 0$, $\mu_*+\alpha_*> 0$.
\begin{itemize}
\item If $q\le 2$, then $j_0=2$, $\theta_1 = s_*\frac{1/q‐1/p_0}{1/p_1‐1/p_0}$, $\theta_2=\frac{\alpha_*(1/p_1‐1/q)+\mu_*(1/p_0‐1/q)}{\gamma_*(1/p_1‐1/p_0)}$.

\item If $q>2$, $p_1\le 2$, then $j_0=2$, $\theta_1=\hat \sigma$, $\theta_2 = \hat \nu$.

\item If $q>2$, $p_1> 2$, then $j_0=4$, $\theta_1=s_*$, $\theta_2=\hat \sigma$, $\theta_3 = \hat \nu$, $\theta_4 = ‐\frac{\mu_*}{\gamma_*}‐\frac 1q +\frac{1}{p_1}$.

\end{itemize}

\item Let $\mu_*+\alpha_*< 0$, $s_*+1/q‐1/p_1<0$.
\begin{itemize}
\item If $q\le 2$, then $j_0=2$, $\theta_1 = s_*\frac{1/q‐1/p_0}{1/p_1‐1/p_0}$, $\theta_2=\hat \theta$.

\item If $q>2$, then $j_0=2$, $\theta_1=\hat \sigma$, $\theta_2 = \frac{q\hat \theta}{2}$.

\end{itemize}

\item Let $\mu_*+\alpha_*< 0$, $s_*+1/q‐1/p_1>0$; if $\gamma_*=0$, we suppose that $\mu_*<0$ and set $‐\mu_*/\gamma_*=+\infty$.
\begin{itemize}
\item If $q\le 2$, then $j_0=3$, $\theta_1 = s_*\frac{1/q‐1/p_0}{1/p_1‐1/p_0}$, $\theta_2=\hat \theta$, $\theta_3 = ‐\frac{\mu_*}{\gamma_*}$.

\item If $q>2$, $p_1\le 2$, then $j_0=3$, $\theta_1=\hat \sigma$, $\theta_2 = ‐\frac{q\mu_*}{2\gamma_*}$, $\theta_3 = \frac{q\hat \theta}{2}$.

\item If $q>2$, $p_1> 2$, then $j_0=5$, $\theta_1=s_*$, $\theta_2=\hat \sigma$, $\theta_3 = ‐\frac{q\mu_*}{2\gamma_*}$, $\theta_4 = \frac{q\hat \theta}{2}$, $\theta_5 = ‐\frac{\mu_*}{\gamma_*}‐\frac 1q+\frac{1}{p_1}$.
\end{itemize}

\end{enumerate}
\end{Not}

\begin{Trm}
\label{main} Let $j_0$, $\theta_j$ $(1\le j\le j_0)$ be defined by Notations \ref{not2}‐‐\ref{not6}.
Suppose that there exists $j_*\in \{1, \, \dots, \, j_0\}$ such that $\theta_{j_*}<\min _{j\ne j_*}\theta_j$, and $\theta_{j_*}>0$. Then 
$$
d_n(M, \, Y_q(\Omega)) \underset{\mathfrak{Z}_0}{\asymp} n^{‐\theta_{j_*}}.
$$
\end{Trm}
\begin{proof}
We define the numbers $\hat m_t$, $\overline{m}_t$, $\tilde m_t$, $m_t$, $m_t'\in \R$ by
\begin{gather}
\label{hat_mt} 2^{\gamma_*k_*t}\cdot 2^{\hat m_t}=n,
\\
\label{line_mt} 2^{\gamma_*k_*t}\cdot 2^{\overline{m}_t}=n^{q/2},
\\
\label{til_mt_t} 2^{\tilde m_t s_*} = 2^{(\mu_* + \alpha_*
+\gamma_*/p_0 -\gamma_*/p_1)k_*t},
\\
\label{mt_t} 2^{m_t(s_*+1/p_0-1/p_1)} = 2^{(\mu_*+\alpha_*) k_*t},
\end{gather}
$$
2^{(\mu_*+\alpha_*)k_*t}\cdot
2^{-(s_*+1/p_0-1/p_1)m_t'}\left( n^{-1/2}\cdot
2^{\gamma_*k_*t/q}\cdot 2^{m_t'/q}\right)
^{\frac{1/p_1-1/p_0}{1/2-1/q}} =1
$$
(the numbers $\overline{m}_t$ and $m'_t$ will be used only for $q>2$).

The numbers $\lambda$ and $\tilde \lambda$ are given by (\ref{lam_til_lam}).

The following equations hold:
\begin{align}
\label{mt_pr_eq} \begin{array}{c} 2^{\mu_*k_*t}\cdot 2^{‐m_t'(s_*+1/q‐1/p_1)}\left(n^{‐1/2}2^{\gamma_*k_*t/q}\cdot 2^{m_t'/q}\right)^{\frac{1/p_1‐1/q}{1/2‐1/q}} = \\ =2^{‐\alpha_*k_*t}\cdot 2^{‐m_t'(1/q‐1/p_0)}\left(n^{‐1/2}2^{\gamma_*k_*t/q}\cdot 2^{m_t'/q}\right)^{\frac{1/p_0‐1/q}{1/2‐1/q}} = \\ = 2^{(\mu_*(1‐\lambda)‐\alpha_*\lambda)k_*t}\cdot 2^{‐m_t'((s_*+1/q‐1/p_1)(1‐\lambda)+(1/q‐1/p_0)\lambda)} = \\ = 2^{(\mu_*(1‐\tilde\lambda)‐\alpha_*\tilde\lambda)k_*t}\cdot 2^{‐m_t'((s_*+1/q‐1/p_1)(1‐\tilde\lambda)+(1/q‐1/p_0)\tilde\lambda)}n^{‐1/2}2^{\gamma_*k_*t/q}\cdot 2^{m_t'/q};\end{array}
\end{align}
\begin{align}
\label{mt_pr_mt_mthatline} \text{if} \quad m_t' = \overline{m}_t, \; \text{then} \; m_t'=m_t; \quad \text{if} \quad m_t' = \hat m_t, \; \text{then} \; m_t'=\tilde{m}_t;
\end{align}
if $m_t=\hat m_t$, then
\begin{align}
\label{mt_hatmt_eq} 2^{\mu_*k_*t}\cdot 2^{‐m_t(s_*+1/q‐1/p_1)} = 2^{‐\alpha_*k_*t}\cdot 2^{‐m_t(1/q‐1/p_0)} \stackrel{\eqref{hat_theta_def}}{=} n^{‐\hat \theta};
\end{align}
if $m_t=\overline{m}_t$, then
\begin{align}
\label{mt_linemt_eq} 2^{\mu_*k_*t}\cdot 2^{‐m_t(s_*+1/q‐1/p_1)} = 2^{‐\alpha_*k_*t}\cdot 2^{‐m_t(1/q‐1/p_0)} \stackrel{\eqref{hat_theta_def}}{=} n^{‐q\hat \theta/2};
\end{align}
if $\tilde m_t=\hat m_t$, then
\begin{align}
\label{tilmt_hatmt_eq} 2^{\mu_*k_*t}\cdot 2^{‐\tilde m_t(s_*+1/q‐1/p_1)}n^{1/q‐1/p_1} = 2^{‐\alpha_*k_*t}\cdot 2^{‐\tilde m_t(1/q‐1/p_0)}n^{1/q‐1/p_0} \stackrel{\eqref{tilde_theta_def}}{=} n^{‐\tilde \theta};
\end{align}
if $m_t'=0$, then
\begin{align}
\label{mtpr0eq} 2^{(\mu_*(1‐\lambda)‐\alpha_*\lambda)k_*t} \stackrel{\eqref{hat_nu_def}}{=} n^{‐\hat \nu};
\end{align}
if $\hat m_t=0$, then
\begin{align}
\label{mthat0eq} 2^{(\mu_*(1‐\tilde\lambda)‐\alpha_*\tilde\lambda)k_*t}\cdot n^{‐\frac 12+\frac 1q} \stackrel{\eqref{tilde_nu_def}}{=} n^{‐\tilde \nu}.
\end{align}
From (\ref{s_1_lam}) it follows that for $p_0>q$, $s_*+\frac{1}{p_0} ‐\frac{1}{p_1}< 0$ we have
\begin{align}
\label{hat_sigma_m0pr} 2^{‐m_0'((1‐\lambda)(s_*+1/q‐1/p_1)+\lambda(1/q‐1/p_0))} = 2^{‐m_0'(1‐\lambda)s_*} \stackrel{(\ref{hat_sigma_def})}{=} n^{‐\hat \sigma}.
\end{align}

\begin{Sta}
\label{m_pr_t0_hat_mt} Let $q>2$, $\hat m_{t_1}=0$, $m'_{t_2}=0$, $\overline{m}_{t_3}=0$. Suppose that one of the following conditions hold:
\begin{enumerate}
\item $p_0<p_1$, $\mu_*+\alpha_*<0<\mu_* +\alpha_* +\frac{\gamma_*}{p_0} ‐\frac{\gamma_*}{p_1}$;
\item $p_0>p_1$, $\mu_* +\alpha_* +\frac{\gamma_*}{p_0} ‐\frac{\gamma_*}{p_1} <0<\mu_*+\alpha_*$.
\end{enumerate}
Then $t_1<t_2<t_3$.
\end{Sta}
\begin{proof}
We have
$$
2^{\gamma_*k_*t_1}=n, \quad 2^{\left(2(\mu_* + \alpha_*)\frac{1/2‐1/q}{1/p_1‐1/p_0}+\frac{2\gamma_*}{q}\right)k_*t_2}=n, \quad 2^{\frac{2\gamma_*k_*t_3}{q}}=n.
$$
In both cases $\frac{\mu_* + \alpha_*}{1/p_1‐1/p_0}>0$; therefore, $(\mu_* + \alpha_*)\frac{1/2‐1/q}{1/p_1‐1/p_0}+\frac{\gamma_*}{q}>\frac{\gamma_*}{q}$ and $t_2<t_3$.

The inequality $t_1<t_2$ is equivalent to $\frac{(\mu_*+\alpha_*)(1/2‐1/q)}{1/p_1‐1/p_0} +\frac{\gamma_*}{q}<\frac{\gamma_*}{2}$, or 
$$
\frac{\mu_*+\alpha_*}{1/p_1‐1/p_0}<\gamma_*.
$$
If $p_0<p_1$, it is equivalent to $\mu_*+\alpha_* +\frac{\gamma_*}{p_0}‐\frac{\gamma_*}{p_1}>0$; if $p_0>p_1$, it is equivalent to $\mu_*+\alpha_* +\frac{\gamma_*}{p_0}‐\frac{\gamma_*}{p_1}<0$.
\end{proof}

The upper estimates for the widths are obtained as in \cite{vas_inters}. Here we write the sketch of the proof.

According to $\mathfrak{Z}_0$, we choose the numbers $\hat t(n)$ and $t_*(n)\in [0, \, \hat t(n)]$ (here $2^{k_*\hat t(n)}$ is a positive degree of $n$, $2^{\gamma_*k_*\hat t(n)} \le n ^{\max\{1, \, q/2\}}$); also we choose a sufficiently small $\varepsilon>0$ (we will write later how to choose these numbers). We set 
\begin{align}
\label{m_t_star}
m_t^* = \max\{\hat m_t‐\varepsilon|t‐t_*(n)|, \, 0\}.
\end{align}

If $q\le 2$ or $\min \{p_0, \, p_1\}\ge q$, we set $k_{t,m}=\nu_{t,m}$ for $0\le m< m_t^*$, $t\le \hat t(n)$, and $k_{t,m}=0$ for $m\ge m_t^*$, $t\le \hat t(n)$.  Then
\begin{align}
\label{s_ktm_le_n}
\begin{array}{c}
\sum \limits _{0\le t\le \hat t(n)} \sum \limits _{m\ge 0} k_{t,m} = \sum \limits _{0\le t\le \hat t(n)} \sum \limits _{0\le m<m_t^*} \nu_{t,m} \stackrel{(\ref{nu_t_m_defin}), (\ref{m_t_star})}{\underset{\mathfrak{Z}_0}{\lesssim}} \sum \limits _{0\le t\le \hat t(n)} 2^{\gamma_*k_*t}\cdot 2^{\hat m_t}\cdot 2^{‐\varepsilon|t‐t_*(n)|} \stackrel{(\ref{hat_mt})}{=}
\\
= \sum \limits _{0\le t\le \hat t(n)} n\cdot 2^{‐\varepsilon|t‐t_*(n)|} \underset{\varepsilon}{\lesssim} n.
\end{array}
\end{align}

Let $q>2$ and $\min \{p_0, \, p_1\}<q$. We choose $m_*(n)\in [(\hat m_{t_*(n)})_+, \, \overline{m}_{t_*(n)}]$ according to $\mathfrak{Z}_0$ and set $k_{t,m}=\nu_{t,m}$ for $0\le m< m_t^*$, $k_{t,m} =\lfloor n\cdot 2^{‐\varepsilon(|t‐t_*(n)|+|m‐m_*(n)|)}\rfloor$ for $0\le t\le \hat t(n)$, $m_t^*\le m\le \overline{m}_t$, $k_{t,m}=0$ for $m>\overline{m}_t$. As above, we get
\begin{align}
\label{s_ktm_le_n0}
\sum \limits _{0\le t\le \hat t(n)} \sum \limits _{m\ge 0} k_{t,m} \underset{\mathfrak{Z}_0,\varepsilon}{\lesssim} n.
\end{align}

Then we apply the following method. Consider the domain
$$
A=A(n)=\{(t,\, m):\; 0\le t\le \hat t(n), \; m\ge (\hat{m}_t)_+\}.
$$
According to Theorems \ref{dn_inters}, \ref{glus}, \ref{p_s}, we write the order estimates of $d_n(W_{t,m}, \, l_q^{\nu_{t,m}})$ for $(t, \, m)\in A\cap \Z_+^2$ (from (\ref{nu_t_m_defin}) and (\ref{hat_mt}) it follows that $\nu_{t,m}\ge 2n$). We obtain the partition of $A$ into polygonal subdomains $A_i=A_i(n)$, $1\le i\le i_0$; for $(t, \, m)\in A_i$, one of the following estimates holds:
\begin{align}
\label{dlwtm_p1}
d_n(W_{t,m}, \, l_q^{\nu_{t,m}}) \underset{\mathfrak{Z}_0}{\asymp} 2^{\mu_*k_*t}\cdot 2^{‐m(s_*+1/q‐1/p_1)} d_n(B_{p_1}^{\nu_{t,m}}, \, l_q^{\nu_{t,m}}),
\end{align}
\begin{align}
\label{dlwtm_p0}
d_n(W_{t,m}, \, l_q^{\nu_{t,m}}) \underset{\mathfrak{Z}_0}{\asymp} 2^{‐\alpha_*k_*t}\cdot 2^{‐m(1/q‐1/p_0)} d_n(B_{p_0}^{\nu_{t,m}}, \, l_q^{\nu_{t,m}}),
\end{align}
\begin{align}
\label{dlwtm_2}
\begin{array}{c}
d_n(W_{t,m}, \, l_q^{\nu_{t,m}}) \underset{\mathfrak{Z}_0}{\asymp} \\ \asymp 2^{((1‐\tilde \lambda)\mu_*‐\tilde \lambda\alpha_*)k_*t}\cdot 2^{‐m((1‐\tilde \lambda)(s_* +1/q‐1/p_1)+\tilde \lambda(1/q‐1/p_0))} d_n(B_2^{\nu_{t,m}}, \, l_q^{\nu_{t,m}}),
\end{array}
\end{align}
\begin{align}
\label{dlwtm_q} d_n(W_{t,m}, \, l_q^{\nu_{t,m}}) \underset{\mathfrak{Z}_0}{\asymp} 2^{((1‐\lambda)\mu_*‐\lambda\alpha_*)k_*t}\cdot 2^{‐m(1‐\lambda)s_*},
\end{align}
where $\lambda$, $\tilde \lambda$ are defined by (\ref{lam_til_lam}).
Taking into account the estimates of  $d_n(B_s^{\nu_{t,m}}, \, l_q^{\nu_{t,m}})$, we get that for $(t, \, m)\in A_i$ the following estimate holds: 
\begin{align}
\label{dnwtm_phi}
d_n(W_{t,m}, \, l_q^{\nu_{t,m}}) \underset{\mathfrak{Z}_0}{\asymp} \varphi_i(t, \, m, \, n) := 2^{\kappa_{1,i}t+\kappa_{2,i}m}n^{\sigma_i}
\end{align}
(here $\kappa_{1,i}$, $\kappa_{2,i}$, $\sigma_i$ are real numbers); in addition, if $(t, \, m)\in A_i\cap A_j$, we have $\varphi_i(t, \, m, \, n) = \varphi_j(t, \, m, \, n)$. Let 
\begin{align}
\label{phi_tmn_def}
\varphi(t, \, m, \, n)=\varphi_i(t, \, m,\, n) \quad \text{for} \quad (t, \, m)\in A_i.
\end{align}

We estimate from above the sum
$$
S:= \sum \limits _{(t, \, m)\in A\cap \Z_+^2} \varphi(t, \, m, \, n) = \sum \limits _{i=1}^{i_0} \sum \limits _{(t, \, m)\in A_i\cap \Z_+^2} \varphi _i(t, \, m, \, n).
$$
Suppose that on each unbounded $A_i$ we have the progression which strictly decreases with $m$. We calculate $\varphi_i(t, \, m, \, n)$ in the vertices of $A_i$, $1\le i\le i_0$. Suppose that they have the form $n^{‐\beta_j}$, $1\le j\le k$, where $\beta_j=\beta_j(\mathfrak{Z}_0)$. If there exists $j_*\in \{1, \, \dots, \, k\}$ such that $\beta _{j_*}< \min _{j\ne j_*} \beta _j$, then
\begin{align}
\label{s_z0_asymp_phi}
S \underset{\mathfrak{Z}_0}{\asymp} n^{‐\beta_{j_*}} = \varphi(t_n, \, m_n, \, n); 
\end{align}
here $(t_n, \, m_n)$ is a vertex of $A_i=A_i(n)$ for some $i$.

Now we set $t_*(n)=t_n$, $m_*(n)=m_n$. Consider the domain 
\begin{align}
\label{a_epsilon}
A^\varepsilon=A^\varepsilon(n)=\{(t,\, m):\; 0\le t\le \hat t(n), \; m\ge m^*_t\}
\end{align}
and its partition into subdomains $A_i^\varepsilon=A_i^\varepsilon(n)\supset A_i$, $1\le i\le i_0$: if $$(\partial A_i)\cap \{(t, \, (\hat m_t)_+):\; t\ge 0\}=\{(t, \, (\hat m_t)_+):\; t_{i,1}\le t\le t_{i,2}\}\ne \varnothing,$$ then $$A_i^\varepsilon = A_i\cup \{(t, \, m):\; t_{i,1}\le t\le t_{i,2}, \; m_t^*\le m\le (\hat m_t)_+\};$$ 
otherwise, $A_i^\varepsilon = A_i$.

For $(t, \, m)\in A_i^\varepsilon$, we estimate $d_{k_{t,m}}(W_{t,m}, \, l_q^{\nu_{t,m}})$ from above as follows: if for $(t, \, m)\in A_i$ we get (\ref{dlwtm_p1}), (\ref{dlwtm_p0}), (\ref{dlwtm_2}) or (\ref{dlwtm_q}), then we write, respectively,
$$
d_{k_{t,m}}(W_{t,m}, \, l_q^{\nu_{t,m}})\underset{\mathfrak{Z}_0}{\lesssim} 2^{\mu_*k_*t}\cdot 2^{‐m(s_*+1/q‐1/p_1)} d_{k_{t,m}}(B_{p_1}^{\nu_{t,m}}, \, l_q^{\nu_{t,m}}),
$$
$$
d_{k_{t,m}}(W_{t,m}, \, l_q^{\nu_{t,m}})\underset{\mathfrak{Z}_0}{\lesssim} 2^{‐\alpha_*k_*t}\cdot 2^{‐m(1/q‐1/p_0)} d_{k_{t,m}}(B_{p_0}^{\nu_{t,m}}, \, l_q^{\nu_{t,m}}),
$$
$$
d_{k_{t,m}}(W_{t,m}, \, l_q^{\nu_{t,m}})\underset{\mathfrak{Z}_0}{\lesssim} 2^{((1‐\tilde \lambda)\mu_*‐\tilde \lambda\alpha_*)k_*t}\cdot 2^{‐m((1‐\tilde \lambda)(s_* +1/q‐1/p_1)+\tilde \lambda(1/q‐1/p_0))}d_{k_{t,m}}(B_2^{\nu_{t,m}}, \, l_q^{\nu_{t,m}}),
$$
$$
d_{k_{t,m}}(W_{t,m}, \, l_q^{\nu_{t,m}}) \underset{\mathfrak{Z}_0}{\lesssim} 2^{((1‐\lambda)\mu_*‐\lambda\alpha_*)k_*t}\cdot 2^{‐m(1‐\lambda)s_*}.
$$

Notice that $2k_{t,m}\le \nu_{t,m}$ for $(t, \, m)\in A^\varepsilon$; indeed, we have $k_{t,m}=0$ or, by (\ref{nu_t_m_defin}), (\ref{hat_mt}), (\ref{m_t_star}), (\ref{a_epsilon}), we have
$$2n\cdot 2^{‐\varepsilon(|t‐t_*(n)|+|m‐m_*(n)|)} \le 2\cdot 2^{\gamma_*k_*t}\cdot 2^{\hat m_t‐\varepsilon|t‐t_*(n)|}.$$

By the definition of $k_{t,m}$, the inequality $2k_{t,m}\le \nu_{t,m}$ and Theorems \ref{glus}, \ref{p_s}, we get that for $(t, \, m)\in A_i^\varepsilon$ $$d_{k_{t,m}}(W_{t,m}, \, l_q^{\nu_{t,m}}) \underset{\mathfrak{Z}_0}{\lesssim}\varphi_i(t, \, m, \, n)\cdot 2^{c_1\varepsilon(|t‐t_*(n)|+|m‐m_*(n)|)},$$ where $c_1=c_1(\mathfrak{Z}_0)$.

If $\varepsilon>0$ is sufficiently small (we choose it according to $\mathfrak{Z}_0$), then
\begin{align}
\label{sum_dktm_wtm_beta}
\sum \limits _{(t, \, m)\in A^\varepsilon \cap \Z_+^2} d_{k_{t,m}}(W_{t,m}, \, l_q^{\nu_{t,m}})\underset{\mathfrak{Z}_0}{\lesssim} n^{‐\beta_*}.
\end{align}

Now we estimate
$\sup _{f\in M} \|f\|_{Y_q(\tilde\Omega_{[\hat t(n)]})}$ according to Proposition \ref{p0_ge_q}. Let $$\sup _{f\in M} \|f\|_{Y_q(\tilde\Omega_{[\hat t(n)]})}\underset{\mathfrak{Z}_0}{\lesssim} n^{‐\beta_j}$$ for some $j\in \{1, \, \dots, \, k\}$. We apply  (\ref{up_est}), (\ref{s_ktm_le_n}), (\ref{s_ktm_le_n0}), (\ref{sum_dktm_wtm_beta}), take into account that $d_{k_{t,m}}(W_{t,m}, \, l_q^{\nu_{t,m}})=0$ for $m<m_t^*$ and get: there is $C=C(\mathfrak{Z}_0)\in \N$ such that
$$
d_{Cn}(M, \, Y_q(\Omega))\underset{\mathfrak{Z}_0}{\lesssim} n^{‐\beta_{j_*}}+n^{‐\beta_j}\underset{\mathfrak{Z}_0}{\lesssim} n^{‐\beta_{j_*}}.
$$
It implies that
$$
d_{n}(M, \, Y_q(\Omega))\underset{\mathfrak{Z}_0}{\lesssim} n^{‐\beta_{j_*}}.
$$

In order to get the lower estimate, we apply (\ref{low_est}). Taking into account that $\nu_{t,m}\ge 2n$ for $(t, \, m)\in A$, we have
$$
d_n(M, \, Y_q(\Omega)) \underset{\mathfrak{Z}_0}{\gtrsim} d_{n}(W_{\lceil t_n\rceil,\lceil m_n\rceil}, \, l_q^{\nu_{\lceil t_n\rceil,\lceil m_n\rceil}}) \stackrel{(\ref{dnwtm_phi}), (\ref{phi_tmn_def})}{\underset{\mathfrak{Z}_0}{\asymp}} \varphi(t_n, \, m_n, \, n) \stackrel{(\ref{s_z0_asymp_phi})}{=} n^{‐\beta_{j_*}}.
$$

Now we consider the cases. We write:
\begin{enumerate}
\item $\hat t(n)$ and the estimate for $\sup _{f\in M}\|f\|_{Y_q(\tilde\Omega_{[\hat t(n)]})}$;
\item the polygonal subdomains $A_i$, $1\le i\le i_0$;
\item the estimates (\ref{dlwtm_p1})‐‐(\ref{dlwtm_q}) for each $A_i$;
\item notice if the progression $\varphi_i(t, \, m, \, n)$ strictly increases or decreases with $m$ or $t$ (if we can see it immediately);
\item the vertices of $A_i$ in which it is sufficient to calculate $\varphi(t, \, m, \, n)$;
\item the values of $\varphi(t, \, m, \, n)$ in these vertices.
\end{enumerate}

First we suppose that $s_*+\frac{1}{\max\{p_0,\, q\}} ‐\frac{1}{p_1}>0$ or $\min\{p_0, \, p_1\}\ge q$. Then from (\ref{sved_k_p1}), (\ref{sved_k_p0}), (\ref{til_mt_t}), (\ref{mt_t}) it follows that for $m\ge \max\{m_t, \, \tilde m_t\}$ we have (\ref{dlwtm_p1}), and for $m\le \min\{m_t, \, \tilde m_t\}$ we have (\ref{dlwtm_p0}).

{\bf Case 1.} Let $\mu_*+\alpha_*\le 0$, $\mu_*+\alpha_*+ \frac{\gamma_*}{p_0}‐\frac{\gamma_*}{p_1}\le 0$. By (\ref{til_mt_t}), (\ref{mt_t}), we have $m_t\le 0$, $\tilde m_t\le 0$ for all $t\ge 0$. Therefore, for all $m\in \Z_+$ we get (\ref{dlwtm_p1}).

First we consider $\gamma_*>0$.

If $p_1\ge q$ or $q\le 2$, then we define $\hat t(n)$ by $\hat m_{\hat t(n)} = 0$. Then 
\begin{align}
\label{supfmmugamma}
\sup _{f\in M}\|f\|_{Y_q(\tilde \Omega _{[\hat t(n)]})} \stackrel{(\ref{nu1}), (\ref{emb_nu}), (\ref{hat_mt})}{\underset{\mathfrak{Z}_0}{\lesssim}} n^{\mu_*/\gamma_*+(1/q‐1/p_1)_+}.
\end{align}
For all $(t, \, m)\in A$ we have (\ref{dlwtm_p1}):
\begin{align}
\label{dnwtm1ql2}
d_n(W_{t,m}, \, l_q^{\nu_{t,m}})\stackrel{(\ref{nu_t_m_defin})}{\underset{\mathfrak{Z}_0}{\asymp}} 2^{\mu_*k_*t}\cdot 2^{‐m(s_*+1/q‐1/p_1)}\cdot 2^{\gamma_*k_*t(1/q‐1/p_1)_+}\cdot 2^{m(1/q‐1/p_1)_+};
\end{align} 
the right‐hand side strictly decreases with $m$. The polygonal domain $A$ has two vertices: $(0, \, \hat m_0)$ and $(\hat t(n), \, 0)$. We substitute these points into (\ref{dnwtm1ql2}) and get $n^{‐s_*+(1/p_1‐1/q)_+}$, $n^{\mu_*/\gamma_*+(1/q‐1/p_1)_+}$. This together with (\ref{supfmmugamma}) yields the order estimates for the widths (see Notation \ref{not2}).

Let $p_1<q$, $q>2$. We define the number $\hat t(n)$ by $\overline{m}_{\hat t(n)} =0$. Then \begin{align}
\label{sup_f_m_qmu}
\sup _{f\in M}\|f\|_{Y_q(\tilde \Omega _{[\hat t(n)]})} \stackrel{(\ref{nu1}), (\ref{emb_nu}), (\ref{line_mt})}{\underset{\mathfrak{Z}_0}{\lesssim}} n^{q\mu_*/2\gamma_*}.
\end{align}
The domain $A$ is divided into two subdomains: $$A_1 = \{(t, \, m):\; 0\le t\le \hat t(n), \; (\hat{m}_t)_+ \le m\le \overline{m}_t\},$$ 
$$A_2 = \{(t, \, m):\; 0\le t\le \hat t(n), \;m\ge \overline{m}_t\}$$ 
(see Fig. 1).

\begin{figure}[h]
\vspace*{-3mm}
\begin{center} \resizebox{70mm}{!}
{\includegraphics{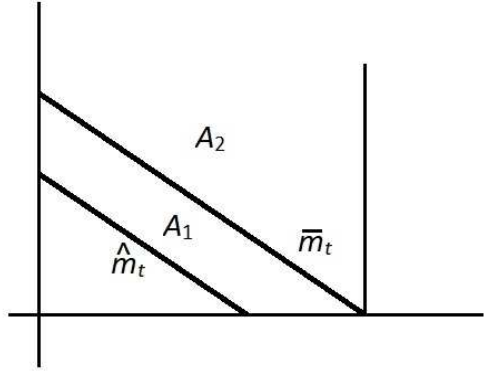}} \vspace*{-3mm}

\parbox[t]{60mm}{\caption{The partition of $A$.}\label{figure1}}
\end{center}
\vspace*{-4mm}
\end{figure}

In both subdomains (\ref{dlwtm_p1}) holds; in $A_2$ the progression strictly decreases with $m$. The vertices of $A_i$ are as follows: $(0, \, \hat m_0)$, $(0, \, \overline{m}_0)$, $(t_1(n), \, 0)$, $(\hat t(n), \, 0)$,  where $t_1(n)$ is defined by $\hat m_{t_1(n)}=0$. We substitute these points into 
\begin{align}
\label{2muktn12}
2^{\mu_*k_*t}\cdot 2^{‐m(s_*+1/q‐1/p_1)}\cdot (n^{‐\frac 12}2^{\frac{\gamma_*k_*t}{q}}2^{\frac mq})^{\lambda _{p_1q}}
\end{align}
(see Theorem \ref{glus}) and get $$n^{‐s_*‐\min(0, \, 1/2‐1/p_1)}, \; n^{‐\frac{q}{2}(s_*+1/q‐1/p_1)}, \; n^{\mu_*/\gamma_*‐\min\{1/p_1‐1/q, \, 1/2‐1/q\}}, \; n^{q\mu_*/2\gamma_*}.$$ By (\ref{sup_f_m_qmu}), we get the desired estimate for the widths (see Notation \ref{not2}).

If $\gamma_*=0$, we choose the number $\hat t(n)$ such that $2^{k_*\hat t(n)}$ is a sufficiently large degree of $n$, and argue as above. Notice that by $\mu_*<0$ the right‐hand side of (\ref{dnwtm1ql2}) and (\ref{2muktn12}) strictly decreases with $t$; hence it is sufficient to substitute $(0, \, \hat m_0)$ into (\ref{dnwtm1ql2}) for $q\le 2$, and $(0, \, \hat m_0)$, $(0, \, \overline{m}_0)$ into (\ref{2muktn12}) for $q>2$.

{\bf Case 2.} Let $p_0\ge q$, $p_1\ge q$. Then for $m\ge \tilde m_t$ we have (\ref{dlwtm_p1}), and for $m\le \tilde m_t$ we have (\ref{dlwtm_p0}). Indeed, let $p_1\le p_0$. Then $m_t\ge \tilde m_t$. If $m\ge m_t$, then (\ref{dlwtm_p1}) follows from (\ref{sved_k_p1}); if $m\le \tilde m_t$, then  (\ref{dlwtm_p0}) follows from (\ref{sved_k_p0}). If $\tilde m_t\le m\le m_t$, by assertion 7 of Theorem \ref{dn_inters} we get (\ref{dlwtm_p1}). Similarly we can consider the case $p_1\ge p_0$.

If $\mu_*+\alpha_*+\frac{\gamma_*}{p_0} ‐\frac{\gamma_*}{p_1}\le 0$, then $\tilde m_t\le 0$. Therefore, (\ref{dlwtm_p1}) holds for all $m\in \Z_+$; in addition, we have (\ref{emb_nu}), where $\nu_*$ is defined by (\ref{nu1}). Hence the widths can be estimated as in Case 1. 

Let $\mu_*+\alpha_*+\frac{\gamma_*}{p_0} ‐\frac{\gamma_*}{p_1}> 0$. Then $\tilde m_t\ge 0$ for $t\ge 0$. We define the number $\hat t(n)$ by $\hat m_{\hat t(n)} = \tilde m_{\hat t(n)}$. Then
$$
A=\{(t, \, m):\; 0\le t\le \hat t(n), \; m\ge \hat m_t\}.
$$
By (\ref{nu3}), (\ref{emb_nu}), (\ref{tilde_theta_def}), (\ref{hat_mt}), (\ref{til_mt_t}), we get
\begin{align}
\label{9tiltheta_est}
\sup _{f\in M}\|f\|_{Y_q(\tilde \Omega _{[\hat t(n)]})} \underset{\mathfrak{Z}_0}{\lesssim} n^{‐\tilde \theta}.
\end{align}

For all $(t, \, m)\in A$ we have (\ref{dlwtm_p1}); the right‐hand side is as follows:
\begin{align}
\label{r_side}2^{\mu_*k_*t}2^{‐m(s_*+1/q‐1/p_1)}2^{(\gamma_*k_*t+m)(1/q‐1/p_1)};
\end{align}
it strictly decreases with $m$. Hence it is sufficient to calculsate (\ref{r_side}) in
$(0, \, \hat m_0)$ and $(\hat t(n), \, \hat m_{\hat t(n)})$. Taking into account (\ref{hat_mt}) and (\ref{tilmt_hatmt_eq}), we get $n^{‐s_*}$ and $n^{‐\tilde \theta}$. This together with (\ref{9tiltheta_est}) yields the desired estimate for $d_n(M, \, Y_q(\Omega))$ (see Notation \ref{not3}).

{\bf Case 3.} Let $\mu_*+\alpha_*< 0$, $\mu_*+\alpha_*+\frac{\gamma_*}{p_0}‐\frac{\gamma_*}{p_1}>0$, $p_0<p_1<q$. Then $\gamma_*>0$, $m_t\le 0$, $\tilde m_t\ge 0$ for $t\ge 0$. In addition, (\ref{emb_nu}) holds with $\nu_*$ defined by (\ref{nu1}).

We apply Theorem \ref{dn_inters} (since $p_0<p_1$, we rearrange 0 and 1).

If $q\le 2$ or $p_1\le 2$, we get (\ref{dlwtm_p1}) for $m\ge m_t$ (see assertion 1 of Theorem \ref{dn_inters}). Hence this estimate holds for all $m\ge 0$. Further we argue as in Case 1.

Let $p_1>2$, $q>2$. The number $\hat t(n)$ is defined by $\overline{m}_{\hat t(n)} = 0$. Then
\begin{align}
\label{sup_fm_qmu2}
\sup _{f\in M}\|f\|_{Y_q(\tilde \Omega _{[\hat t(n)]})} \stackrel{(\ref{nu1}), (\ref{emb_nu}), (\ref{line_mt})}{\underset{\mathfrak{Z}_0}{\lesssim}} n^{q\mu_*/2\gamma_*}.
\end{align}

We define the points $t_1(n)$, $t_2(n)$, $t_3(n)$ by $\hat m_{t_1(n)}=\tilde m_{t_1(n)}$, $m'_{t_2(n)}=0$, $\hat m_{t_3(n)}=0$. By Proposition \ref{m_pr_t0_hat_mt}, we have $t_3(n)<t_2(n)<\hat t(n)$.

Taking into account (\ref{mt_pr_mt_mthatline}), we get that $A$ is divided into
$$
A_1=\{(t, \, m):\; 0\le t\le \hat t(n), \; (\hat m_t)_+\le m\le \overline{m}_t, \; m\ge m_t'\},
$$
$$
A_2=\{(t, \, m):\; 0\le t\le \hat t(n), \; m\ge \overline{m}_t\},
$$
$$
A_3=\{(t, \, m):\; 0\le t\le \hat t(n), \; (\hat m_t)_+\le m\le m_t'\}
$$
(see Fig. 2).

\begin{figure}[h]
\vspace*{-3mm}
\begin{center} \resizebox{70mm}{!}
{\includegraphics{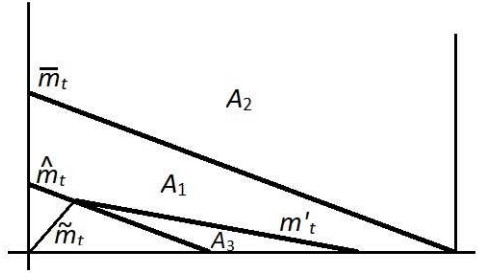}} \vspace*{-3mm}

\parbox[t]{60mm}{\caption{The partition of $A$.}\label{figure2}}
\end{center}
\vspace*{-4mm}
\end{figure}

For $(t, \, m)\in A_1\cup A_2$ we get (\ref{dlwtm_p1}) by (\ref{sved_k_p0}) and assertions 2, 3 of Theorem \ref{dn_inters} (recall that we rearrange 0 and 1); for $(t, \, m)\in A_3$ we have (\ref{dlwtm_p0}) if $p_0\ge 2$, and (\ref{dlwtm_2}) if $p_0<2$. In $A_2$ the progression strictly decreases with $m$. If $p_0\ge 2$, then in $A_3$ the progression strictly increases with $m$.

Hence, for $p_0\ge 2$ it is sufficient to calculate the right‐hand side of (\ref{dlwtm_p1}) in $(0, \, \hat m_0)$, $(0, \, \overline{m}_0)$, $(\hat t(n), \, 0)$, $(t_1(n), \, \hat m_{t_1}(n))$, $(t_2(n), \, 0)$, and for $p_0<2$, the same values and the right‐hand side of (\ref{dlwtm_2}) in $(t_3(n), \, 0)$. Applying (\ref{hat_mt}), (\ref{line_mt}), (\ref{mt_pr_eq}), (\ref{tilmt_hatmt_eq}), (\ref{mtpr0eq}), (\ref{mthat0eq}), we get $n^{‐\theta_j}$, $1\le j\le j_0$ (see Notation \ref{not4}, case 1, subcase $q>2$, $2< p_1<q$). This together with (\ref{sup_fm_qmu2}) yields the desired estimate for the widths.

{\bf Case 4.} Let $\mu_*+\alpha_*< 0$, $\mu_*+\alpha_*+\frac{\gamma_*}{p_0}‐\frac{\gamma_*}{p_1}>0$, $p_0<q<p_1$. Again we get $\gamma_*>0$, $\tilde m_t\ge 0$, $m_t\le 0$ for $t\ge 0$.

Let $q\le 2$. We define the number $\hat t(n)$ by $\hat m_{\hat t(n)}=0$. 
Then
\begin{align}
\label{theta3lam}
\sup _{f\in M}\|f\|_{Y_q(\tilde \Omega _{[\hat t(n)]})} \stackrel{(\ref{nu2}), (\ref{emb_nu})}{\underset{\mathfrak{Z}_0}{\lesssim}} 2^{((1‐\lambda)\mu_*‐\lambda \alpha_*)k_*\hat t(n)} \stackrel{(\ref{hat_mt})}{=} n^{((1‐\lambda)\mu_*‐\lambda \alpha_*)/\gamma_*} = n^{‐\theta _3}
\end{align}
(see Notation \ref{not4}, case 2, subcase $q\le 2$).

We define the number $t_1(n)$ by $\hat m_{t_1(n)} = \tilde m_{t_1(n)}$.

The domain $A$ is divided into
$$
A_1=\{(t, \, m):\; 0\le t\le \hat t(n), \; m\ge (\hat m_t)_+, \; m\ge \tilde m_t\},
$$
$$
A_2=\{(t, \, m):\; 0\le t\le \hat t(n), \; (\hat m_t)_+\le m\le \tilde m_t\}.
$$

\begin{figure}[h]
\vspace*{-3mm}
\begin{center} \resizebox{50mm}{!}
{\includegraphics{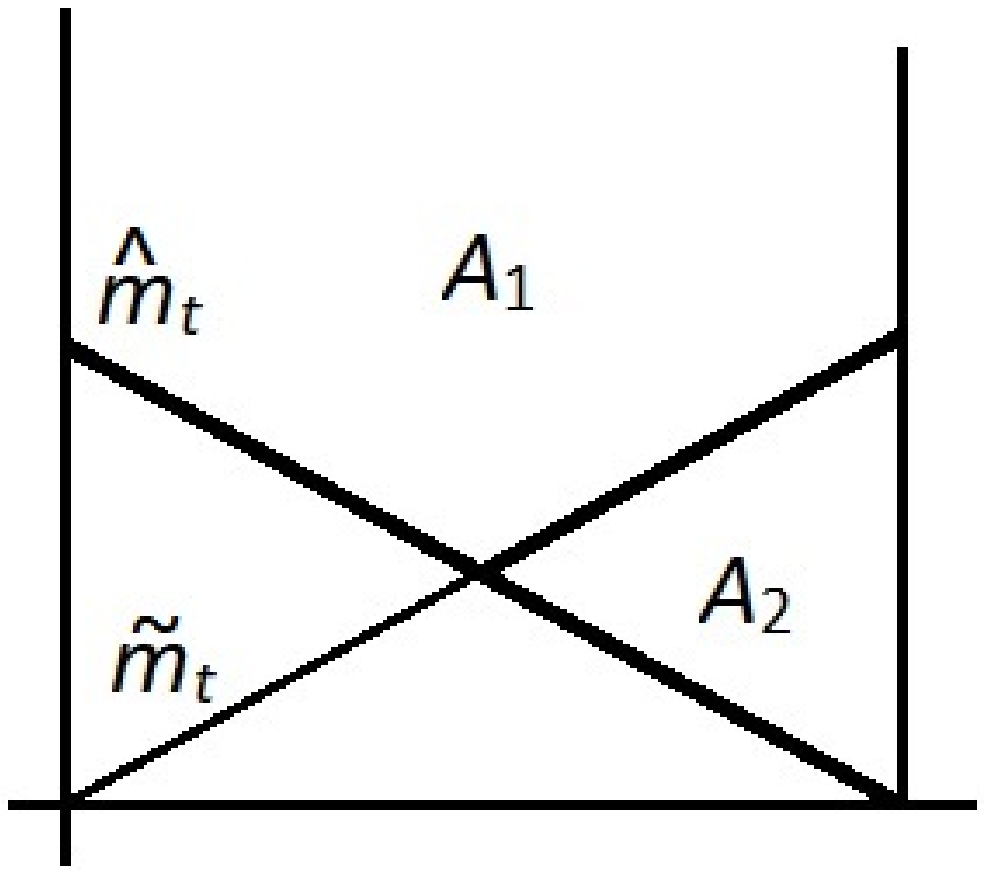}} \vspace*{-3mm}

\parbox[t]{60mm}{\caption{The partition of $A$.}\label{figure3}}
\end{center}
\vspace*{-4mm}
\end{figure}

By (\ref{sved_k_p0}) and assertion 6 of Theorem \ref{dn_inters} (with rearranged 0 and 1), we get (\ref{dlwtm_p1}) in $A_1$, and (\ref{dlwtm_q}) in $A_2$. Everywhere the progression strictly decreases with $m$. Hence it is sufficient to calculate $2^{\mu_*k_*t}\cdot 2^{‐m(s_*+1/q‐1/p_1)}\cdot (2^{\gamma_*k_*t}\cdot 2^m)^{1/q‐1/p_1}$ in $(0, \, \hat m_0)$ and $(t_1(n), \, \hat m_{t_1(n)})$, and $2^{((1‐\lambda)\mu_*‐\lambda\alpha_*)k_*t}\cdot 2^{‐m(1‐\lambda)s_*}$ in $(\hat t(n), \, 0)$. Applying (\ref{hat_mt}) and (\ref{tilmt_hatmt_eq}), we get $n^{‐\theta_j}$, $1\le j\le 3$ (see Notation \ref{not4}, case 2, subcase $q\le 2$). This together with (\ref{theta3lam}) yields the desired estimate for the widths.

Let $q>2$. We define $\hat t(n)$ by $m'_{\hat t(n)}=0$. Then $2^{\gamma_*k_*\hat t(n)}\le n^{q/2}$ (see Proposition \ref{m_pr_t0_hat_mt}),
\begin{align}
\label{9est_nu}
\sup _{f\in M}\|f\|_{Y_q(\tilde \Omega _{[\hat t(n)]})} \stackrel{(\ref{nu2}), (\ref{emb_nu})}{\underset{\mathfrak{Z}_0}{\lesssim}} 2^{((1‐\lambda)\mu_*‐\lambda \alpha_*)k_*\hat t(n)} \stackrel{(\ref{mtpr0eq})}{=} n^{‐\hat \nu}.
\end{align}

We define the numbers $t_1(n)$ and $t_2(n)$ by $\hat m_{t_1(n)}=\tilde m_{t_1(n)}$ and $\hat m_{t_2(n)}=0$.
By Proposition \ref{m_pr_t0_hat_mt}, $t_2(n)<\hat t(n)$.

Taking into account (\ref{mt_pr_mt_mthatline}), we get that $A$ is divided into
$$
A_1 = \{(t, \, m):\; 0\le t\le \hat t(n), \; m\ge \tilde m_t, \; m\ge (\hat m_t)_+\},
$$
$$
A_2 = \{(t, \, m):\; 0\le t\le \hat t(n), \; m'_t\le m\le \tilde m_t\},
$$
$$
A_3 = \{(t, \, m):\; 0\le t\le \hat t(n), \; (\hat m_t)_+\le m\le \tilde m_t, \; m\le m'_t\}
$$
(see Fig. 4).

\begin{figure}[h]
\vspace*{-3mm}
\begin{center} \resizebox{60mm}{!}
{\includegraphics{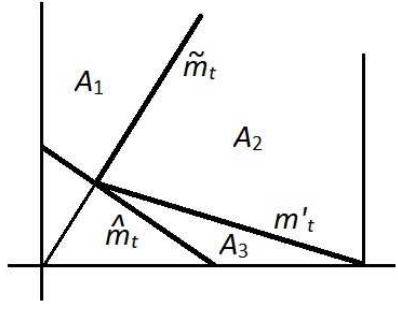}} \vspace*{-3mm}

\parbox[t]{60mm}{\caption{The partition of $A$.}\label{figure4}}
\end{center}
\vspace*{-4mm}
\end{figure}

We apply Theorem \ref{dn_inters} (assertions 4‐‐5 and (\ref{sved_k_p0})); since $p_0<p_1$, we rearrange 0 and 1. We get (\ref{dlwtm_p1}) in $A_1$ and (\ref{dlwtm_q}) in $A_2$. The progressions strictly decrease with $m$. If $p_0\ge 2$, we get (\ref{dlwtm_p0}) in $A_3$; the progression strictly increases with $m$. If $p_0<2$, we get (\ref{dlwtm_2}) in $A_3$.

For $p_0\ge 2$ it is sufficient to calculate $$2^{\mu_*k_*t}\cdot 2^{‐m(s_*+1/q‐1/p_1)}\cdot (2^{\gamma_*k_*t}\cdot 2^m)^{1/q‐1/p_1}$$ in $(0, \, \hat m_0)$, $(t_1(n), \, \hat m_{t_1(n)})$ and $2^{((1‐\lambda)\mu_*‐\lambda\alpha_*)k_*t}\cdot 2^{‐m(1‐\lambda)s_*}$ in $(\hat t(n), \, 0)$. For $p_0<2$, in addition, we calculate the right‐hand side of (\ref{dlwtm_2}) in $(t_2(n), \, 0)$. Applying (\ref{hat_mt}), (\ref{mt_pr_eq}), (\ref{tilmt_hatmt_eq}), (\ref{mtpr0eq}), (\ref{mthat0eq}), we get $n^{‐\theta_j}$, $1\le j\le j_0$ (see Notation \ref{not4}, case 2, subcase $q>2$). This together with (\ref{9est_nu}) yields the desired estimates.

{\bf Case 5.} Let $\mu_*+\alpha_*> 0$, $\mu_*+\alpha_*+\frac{\gamma_*}{p_0}‐\frac{\gamma_*}{p_1}< 0$, $p_1<p_0<q$. Then $\gamma_*>0$, $\tilde m_t\le 0$, $m_t\ge 0$ for $t\ge 0$.

If $q\le 2$, then we define the number $\hat t(n)$ by $\hat m_t(n)=m_t(n)$. From (\ref{nu4}), (\ref{emb_nu}), (\ref{hat_theta_def}), (\ref{hat_mt}), (\ref{mt_t}) we get
\begin{align}
\label{supfmthetahat}
\sup _{f\in M}\|f\|_{Y_q(\tilde \Omega _{[\hat t(n)]})} \underset{\mathfrak{Z}_0}{\lesssim} n^{‐\hat \theta}.
\end{align}
From (\ref{sved_k_p1}) it follows that (\ref{dlwtm_p1}) holds for all $(t, \, m)\in A$. The right‐hand side is equal to $2^{\mu_*k_*t}\cdot 2^{‐m(s_*+1/q‐1/p_1)}$ and strictly decreases with $m$. Therefore, it is sufficient to calculate this value in $(0, \, \hat m_0)$ and $(\hat t(n), \, \hat m_{\hat t(n)})$. Taking into account (\ref{hat_mt}) and (\ref{mt_hatmt_eq}), we get $n^{‐s_*‐1/q+1/p_1}$ and $n^{‐\hat\theta}$. This together with (\ref{supfmthetahat}) yields the desired estimate for the widths (see Notation \ref{not5}, case 1, subcase $q\le 2$).

Let $q>2$. We define the number $\hat t(n)$ by $\overline{m}_{\hat t(n)}=m_{\hat t(n)}$. From (\ref{nu4}), (\ref{emb_nu}), (\ref{hat_theta_def}), (\ref{line_mt}), (\ref{mt_t}) we get
\begin{align}
\label{qth2sup}
\sup _{f\in M}\|f\|_{Y_q(\tilde \Omega _{[\hat t(n)]})} \underset{\mathfrak{Z}_0}{\lesssim} n^{‐q\hat \theta/2}.
\end{align}

Let $p_0\le 2$. We define the number $t_1(n)$ by $\hat m_{t_1(n)} = m_{t_1(n)}$.

The domain $A$ is divided into
$$
A_1 = \{(t, \, m):\; 0\le t\le \hat t(n), \; (\hat m_t)_+\le m\le \overline{m}_t, \; m\ge m_t\},
$$
$$
A_2 = \{(t, \, m):\; 0\le t\le \hat t(n), \; m\ge \overline{m}_t\},
$$
$$
A_3 = \{(t, \, m):\; 0\le t\le \hat t(n), \; m\ge (\hat m_t)_+, \; m\le m_t\}
$$
(see Fig. 5).

\begin{figure}[h]
\vspace*{-3mm}
\begin{center} \resizebox{50mm}{!}
{\includegraphics{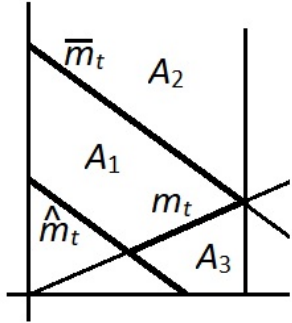}} \vspace*{-3mm}

\parbox[t]{60mm}{\caption{The partition of $A$.}\label{figure5}}
\end{center}
\vspace*{-4mm}
\end{figure}

From (\ref{sved_k_p1}) it follows that (\ref{dlwtm_p1}) holds for $(t, \, m)\in A_1\cup A_2$; in $A_2$ the progression strictly decreases with $m$; from assertion 1 of Theorem \ref{dn_inters} it follows that (\ref{dlwtm_p0}) holds in $A_3$ (the progression strictly increases with $m$). Hence it is sufficient to calculate $2^{\mu_*k_*t}\cdot 2^{‐m(s_*+1/q‐1/p_1)}\cdot n^{‐1/2}2^{\gamma_*k_*t/q}2^{m/q}$ in $(0, \, \hat m_0)$, $(0, \, \overline{m}_0)$, $(t_1(n), \, \hat m_{t_1(n)})$ and $(\hat t(n), \, \overline{m}_{\hat t(n)})$. Taking into account (\ref{hat_mt}), (\ref{line_mt}), (\ref{mt_hatmt_eq}) and (\ref{mt_linemt_eq}), we get $n^{‐\theta_j}$, $1\le j\le 4$ (see Notation \ref{not5}, case 1, subcase $q>2$, $p_0\le 2$).

Let $q> 2$, $p_0> 2$. 

If $p_1\ge 2$, we define the numbers $t_1(n)$ and $t_2(n)$ by $\hat m_{t_1(n)}=0$ and $m'_{t_2(n)}=0$. By Proposition \ref{m_pr_t0_hat_mt}, we have $t_1(n)<t_2(n)$.

The domain $A$ is divided into
$$
A_1 = \{(t, \, m):\; 0\le t\le \hat t(n), \; (\hat m_t)_+\le m\le \overline{m}_t; \; m\ge m'_t\text{ for }t_2(n)<t\le \hat t(n)\},
$$
$$
A_2 = \{(t, \, m):\; 0\le t\le \hat t(n), \; m\ge \overline{m}_t\},
$$
$$
A_3 = \left\{ \begin{array}{l}\{(t, \, m):\; 0\le t\le \hat t(n), \; 0\le m\le m'_t\} \quad \text{if }m'_t\uparrow\uparrow, \\ \varnothing \quad \text{otherwise}\end{array}\right.
$$
(see Fig. 6).

\begin{figure}[h]
\vspace*{-3mm}
\begin{center} \resizebox{50mm}{!}
{\includegraphics{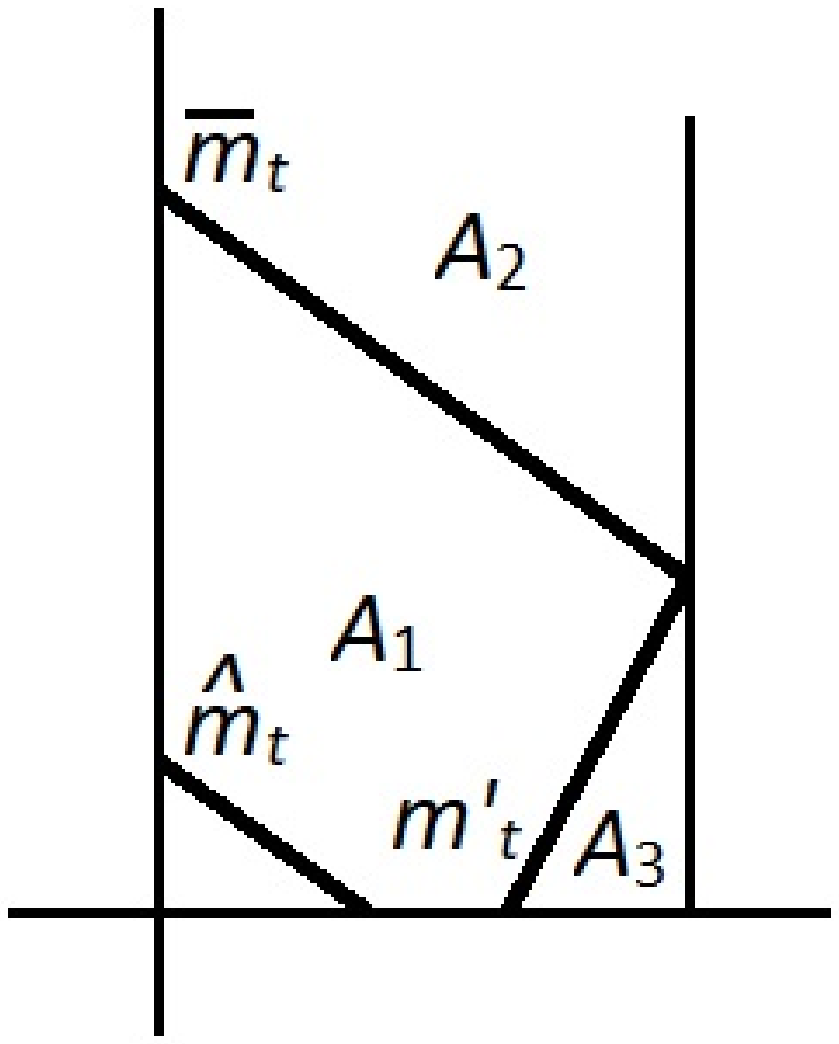}} \vspace*{-3mm}

\parbox[t]{70mm}{\caption{The partition of $A$.}\label{figure6}}
\end{center}
\vspace*{-4mm}
\end{figure}

By Theorem \ref{dn_inters} (see (\ref{sved_k_p1}) and assertion 3), (\ref{dlwtm_p1}) holds in $A_1\cup A_2$ (in $A_2$ the progression strictly decreases with $m$), (\ref{dlwtm_p0}) holds in $A_3$ (the progression strictly increases with $m$). Hence it is sufficient to calculate the right‐hand side of (\ref{dlwtm_p1}) in $(0, \, \hat m_0)$, $(0, \, \overline{m}_0)$, $(t_1(n), \, 0)$, $(t_2(n), \, 0)$ and $(\hat t(n), \, \overline{m}_{\hat t(n)})$. Taking into account (\ref{hat_mt}), (\ref{line_mt}), (\ref{mt_pr_eq}), (\ref{mt_linemt_eq}) and (\ref{mtpr0eq}), we get $n^{‐\theta_j}$, $1\le j\le 5$ (see Notation \ref{not5}, case 1, subcase $q>2$, $p_1\ge 2$).

Let $p_1<2$. We define the numbers $t_1(n)$, $t_2(n)$ and $t_3(n)$ by $m_{t_1(n)} = \hat m_{t_1(n)}$, $\hat m_{t_2(n)}=0$ and $m'_{t_3(n)}=0$. By Proposition \ref{m_pr_t0_hat_mt}, we have $t_2(n)<t_3(n)$.

The domain $A$ is divided into subsets
$$
A_1 = \{(t, \, m):\; 0\le t\le \hat t(n), \; (\hat m_t)_+\le m\le \overline{m}_t, \; m\ge m_t\},
$$
$$
A_2 = \{(t, \, m):\; 0\le t\le \hat t(n), \; m\ge \overline{m}_t\},
$$
$$
A_3 = \{(t, \, m):\; 0\le t\le \hat t(n), \; (\hat m_t)_+\le m\le m_t; \; m\ge m'_t \text{ for }t_2(n)<t\le \hat t(n)\},
$$
$$
A_4 = \left\{ \begin{array}{l}\{(t, \, m):\; 0\le t\le \hat t(n), \; 0\le m\le m'_t\} \quad \text{if }m'_t\uparrow\uparrow, \\ \varnothing \quad\text{otherwise}\end{array}\right.
$$
(see Fig. 7).

\begin{figure}[h]
\vspace*{-3mm}
\begin{center} \resizebox{50mm}{!}
{\includegraphics{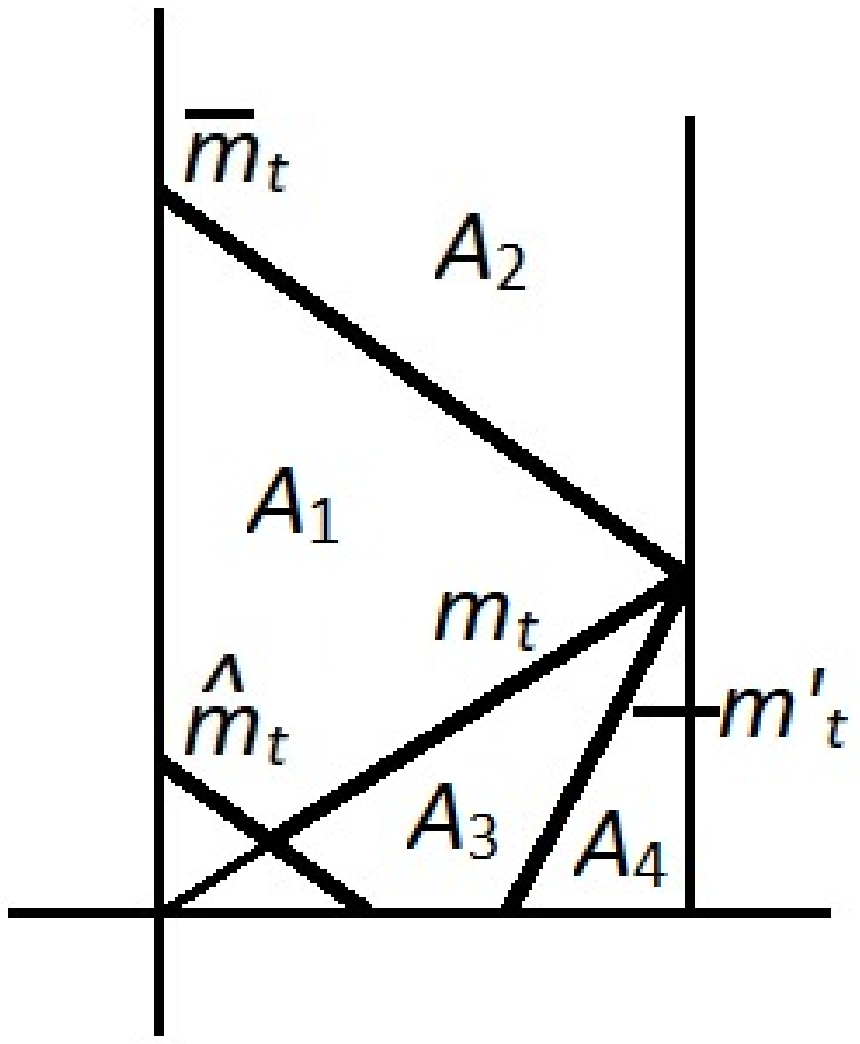}} \vspace*{-3mm}

\parbox[t]{60mm}{\caption{The partition of $A$.}\label{figure7}}
\end{center}
\vspace*{-4mm}
\end{figure}

From (\ref{sved_k_p1}) and assertion 2 of Theorem \ref{dn_inters} it follows that (\ref{dlwtm_p1}) holds in $A_1\cup A_2$ (in $A_2$ the progression strictly decreases with $m$), (\ref{dlwtm_2}) holds in $A_3$, (\ref{dlwtm_p0}) holds in $A_4$ (the progression strictly increases with $m$). Hence it is sufficient to calculate the right‐hand side of (\ref{dlwtm_p1}) in $(0, \, \hat m_0)$, $(0, \, \overline{m}_0)$, $(t_1(n), \, \hat m_{t_1(n)})$ and $(\hat t(n), \, \overline{m}_{\hat t(n)})$, and the right‐hand side of (\ref{dlwtm_2}) in $(t_2(n), \, 0)$ and $(t_3(n), 0)$. Taking into account (\ref{hat_mt}), (\ref{line_mt}), (\ref{mt_t}), (\ref{mt_pr_eq}), (\ref{mt_hatmt_eq}), (\ref{mt_linemt_eq}), (\ref{mtpr0eq}), (\ref{mthat0eq}), we get $n^{‐\theta_j}$, $1\le j\le 6$ (see Notation \ref{not5}, case 1, subcase $q>2$, $p_1< 2<p_0$).

This together with (\ref{qth2sup}) yields the desired estimate for the widths.

{\bf Case 6.} Let $\mu_*+\alpha_*> 0$, $\mu_*+\alpha_*+\frac{\gamma_*}{p_0}‐\frac{\gamma_*}{p_1}< 0$, $p_1<q<p_0$. Then $\gamma_*>0$, $\tilde m_t\le 0$, $m_t\ge 0$ for $t\ge 0$.

Let $q\le 2$. We define the number $\hat t(n)$ by $\hat m_{\hat t(n)}=0$. By (\ref{nu2}) and (\ref{emb_nu}), we have (\ref{theta3lam}); the number $\theta_3$ is the same. In addition, we define $t_1(n)$ by $m_{t_1(n)} = \hat m_{t_1(n)}$. The domain $A$ is divided into
$$
A_1 = \{(t, \, m):\; 0\le t\le \hat t(n), \; m\ge (\hat m_t)_+, \; m\ge m_t\},
$$
$$
A_2 = \{(t, \, m):\; 0\le t\le \hat t(n), \; m\ge (\hat m_t)_+, \; m\le m_t\}.
$$
By (\ref{sved_k_p1}) and assertion 6 of Theorem \ref{dn_inters}, (\ref{dlwtm_p1}) holds in $A_1$, (\ref{dlwtm_q}) holds in $A_2$ (in both domains the progression strictly decreases with $m$). Therefore it is sufficient to calculate $2^{\mu_*k_*t}\cdot 2^{‐m(s_*+1/q‐1/p_1)}$ in $(0, \, \hat m_0)$ and $(t_1(n), \, \hat m_{t_1(n)})$, and $2^{((1‐ \lambda)\mu_*‐ \lambda\alpha_*)k_*t}\cdot 2^{‐m(1‐\lambda)s_*}$ in $(\hat t(n), \, 0)$. Taking into account (\ref{hat_mt}) and (\ref{mt_hatmt_eq}), we get $n^{‐\theta_j}$, $1\le j\le 3$ (see Notation \ref{not5}, case 2, subcase $q\le 2$). This together with (\ref{theta3lam}) yields the desired estimate for the widths.

Let $q> 2$. We define the number $\hat t(n)$ by $m'_{\hat t(n)}=0$. Then (\ref{9est_nu}) holds. By Proposition \ref{m_pr_t0_hat_mt}, we have  $2^{\gamma_*k_*\hat t(n)}\le n^{q/2}$.

Let $p_1\ge 2$. We define the numbers $t_1(n)$ and $t_2(n)$ by $m_{t_1(n)} = \overline{m}_{t_1(n)}$, $\hat m_{t_2(n)}=0$. By Proposition \ref{m_pr_t0_hat_mt}, we have $t_2(n)<\hat t(n)$.

The domain $A$ is divided into
$$
A_1 = \{(t, \, m):\; 0\le t\le \hat t(n), \; (\hat m_t)_+\le m\le \overline{m}_t; \; m\le m'_t\text{ for }t_1(n)<t\le \hat t(n)\},
$$
$$
A_2 = \{(t, \, m):\; 0\le t\le \hat t(n), \; m\ge \overline{m}_t, \; m\ge m_t\},
$$
$$
A_3 = \left\{\begin{array}{l}\{(t, \, m):\;0\le t\le \hat t(n), \; m'_t\le m \le m_t\}, \quad \text{if }m'_t\downarrow\downarrow, \\ \varnothing, \quad \text{otherwise}\end{array}\right.
$$
(see Fig. 8).

\begin{figure}[h]
\vspace*{-3mm}
\begin{center} \resizebox{50mm}{!}
{\includegraphics{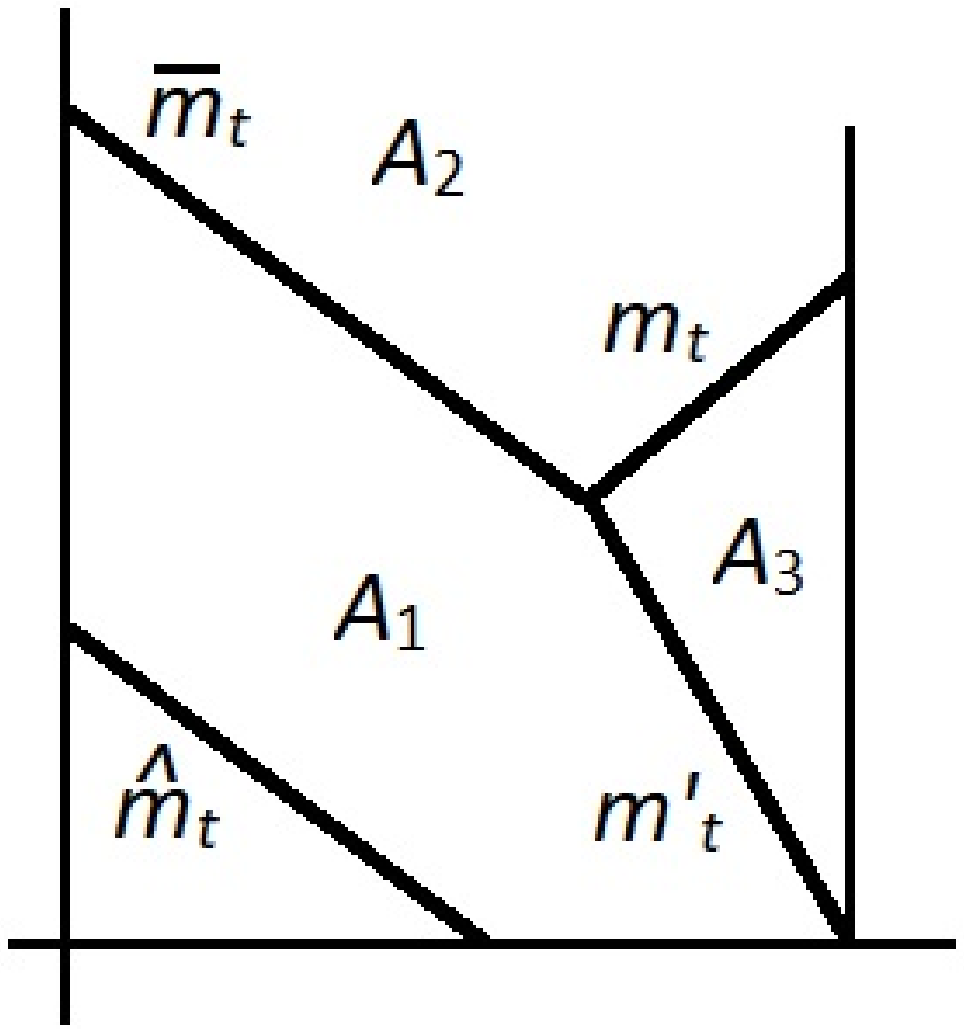}} \vspace*{-3mm}

\parbox[t]{60mm}{\caption{The partition of $A$.}\label{figure8}}
\end{center}
\vspace*{-4mm}
\end{figure}

By (\ref{sved_k_p1}) and assertion 4 of Theorem \ref{dn_inters}, (\ref{dlwtm_p1}) holds in $A_1\cup A_2$ (in $A_2$ the progression strictly decreases with $m$), (\ref{dlwtm_q}) holds in $A_3$ (the progression strictly decreases with $m$). Therefore it is sufficient to calculate the right‐hand side of (\ref{dlwtm_p1}) in $(0, \, \hat m_0)$, $(0, \overline{m}_0)$, $(t_1(n), \, \overline{m}_{t_1(n)})$, $(t_2(n), \, 0)$ and $(\hat t(n), \, 0)$. Taking into account (\ref{hat_mt}), (\ref{line_mt}), (\ref{mt_pr_eq}), (\ref{mt_linemt_eq}),  (\ref{mtpr0eq}), we get $n^{‐\theta_j}$, $1\le j\le 5$ (see Notation \ref{not5}, case 2, subcase $q> 2$, $p_1\ge 2$). This together with (\ref{9est_nu}) yields the desired estimate for the widths.

Let $p_1< 2$. We define the numbers $t_1(n)$, $t_2(n)$ and $t_3(n)$ by $m_{t_1(n)} = \overline{m}_{t_1(n)}$, $m_{t_2(n)} = \hat{m}_{t_2(n)}$, $\hat m_{t_3(n)}=0$. By Proposition \ref{m_pr_t0_hat_mt}, we have $t_3(n)<\hat t(n)$.

The domain $A$ is divided into
$$
A_1 = \{(t, \, m):\; 0\le t\le \hat t(n), \; (\hat m_t)_+\le m\le \overline{m}_t, \; m\ge m_t\},
$$
$$
A_2 = \{(t, \, m):\; 0\le t\le \hat t(n), \; m\ge \overline{m}_t, \; m\ge m_t\},
$$
$$
A_3 = \{(t, \, m):\; 0\le t\le \hat t(n), \; m\ge (\hat{m}_t)_+, \; m\le m_t; \; m\le m_t' \text{ for }t_1(n)<t\le \hat t(n)\},
$$
$$
A_4 = \left\{\begin{array}{l}\{(t, \, m):\; 0\le t\le \hat t(n), \; m'_t\le m\le m_t\} \quad \text{if }m'_t\downarrow\downarrow, \\ \varnothing \quad\text{otherwise} \end{array} \right.
$$
(see Fig. 9).

\begin{figure}[h]
\vspace*{-3mm}
\begin{center} \resizebox{50mm}{!}
{\includegraphics{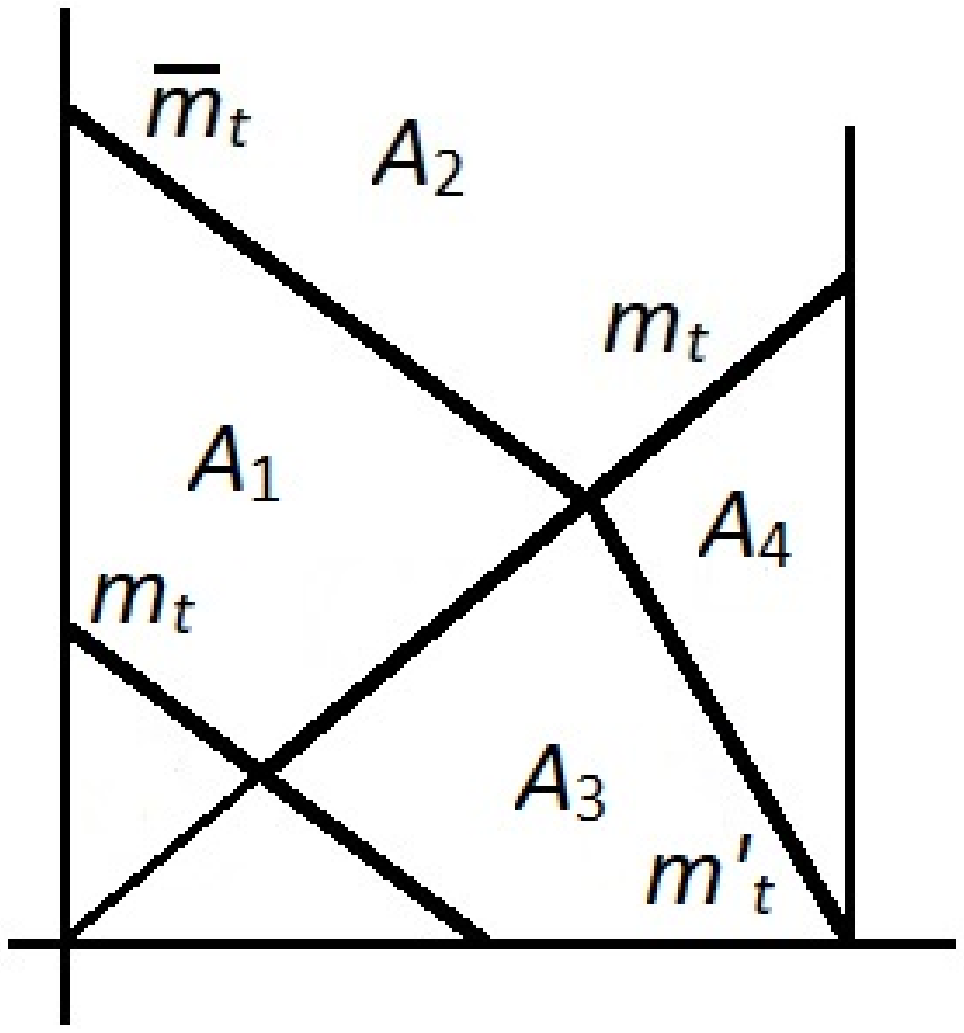}} \vspace*{-3mm}

\parbox[t]{60mm}{\caption{The partition of $A$.}\label{figure9}}
\end{center}
\vspace*{-4mm}
\end{figure}

By (\ref{sved_k_p1}) and assertion 5 of Theorem \ref{dn_inters}, (\ref{dlwtm_p1}) holds in $A_1\cup A_2$ (in $A_2$ the progression strictly decreases with $m$), (\ref{dlwtm_2}) holds in $A_3$, (\ref{dlwtm_q}) holds in $A_4$ (the progression strictly decreases with $m$). Therefore, it is sufficient to calculate the right‐hand side of (\ref{dlwtm_p1}) in $(0, \, \hat m_0)$, $(0, \overline{m}_0)$, $(t_1(n), \, \overline{m}_{t_1(n)})$, $(t_2(n), \, \hat m_{t_2(n)})$ and the right‐hand side of (\ref{dlwtm_2}) in $(t_3(n), \, 0)$ and $(\hat t(n), \, 0)$. Taking into account (\ref{hat_mt}), (\ref{line_mt}), (\ref{mt_pr_eq}), (\ref{mt_hatmt_eq}), (\ref{mt_linemt_eq}),  (\ref{mtpr0eq}), (\ref{mthat0eq}), we get $n^{‐\theta_j}$, $1\le j\le 6$ (see Notation \ref{not5}, case 2, subcase $q> 2$, $p_1< 2$).

This together with (\ref{9est_nu}) yields the desired estimate of the widths.

\vskip 0.3cm Now we consider the case $p_0>q>p_1$, $s_* +\frac{1}{p_0}‐\frac{1}{p_1}< 0$. Notice that if $q>2$, then by $s_*>0$ and $s_*+\frac{1}{p_0}‐\frac{1}{p_1}< 0$ we get 
\begin{align}
\label{2hm0m0provrmo}
2^{\hat m_0}=n<2^{m'_0}\le n^{q/2}=2^{\overline{m}_0}.\end{align}

Notice that 
\begin{gather}
\label{s30} 2^{(\mu_*+\alpha_*)k_*t}\cdot 2^{‐m(s_*+1/p_0‐1/p_1)}\ge 1 \; \Leftrightarrow \; m\ge m_t,
\\
\label{s31} 2^{(\mu_*+\alpha_*+\gamma_*/p_0‐\gamma_*/p_1)k_*t}\cdot 2^{‐ms_*}\le 1 \; \Leftrightarrow \; m\ge \tilde m_t,
\\
\label{s32} n^{1/2}(2^{\gamma_*k_*t}\cdot 2^m)^{‐1/q}\ge (2^{(\mu_*+\alpha_*)k_*t}\cdot 2^{‐m(s_*+1/p_0‐1/p_1)})^{\frac{1/2‐1/q}{1/p_1‐1/p_0}} \; \Leftrightarrow \; m\le m_t'.
\end{gather}
 
{\bf Case 1.} Let $\mu_*+\alpha_*+\gamma_*/p_0‐\gamma_*/p_1\ge 0$. Then $\tilde m_t\ge 0$, $m_t\le 0$ for $t\ge 0$. We define the number $\hat t(n)$ by $\hat m_{\hat t(n)} =\tilde m_{\hat t(n)}$. By (\ref{nu3}), (\ref{emb_nu}), (\ref{tilde_theta_def}), (\ref{hat_mt}), (\ref{til_mt_t}), we get (\ref{9tiltheta_est}).

Let $q\le 2$. If $(t, \, m)\in A$, we have $m\ge \tilde m_t$. By (\ref{s30}), (\ref{s31}) and assertion 6 of Theorem \ref{dn_inters}, we have (\ref{dlwtm_q}); the progression strictly decreases with $m$. Hence it is sufficient to calculate $2^{((1‐\lambda)\mu_*‐\lambda\alpha_*)k_*t}\cdot 2^{‐m(1‐\lambda)s_*}$ in $(0, \, \hat m_0)$ and $(\hat t(n), \, \hat m_{\hat t(n)})$. Applying (\ref{tilmt_hatmt_eq}), we get $n^{‐s_*\frac{1/q‐1/p_0}{1/p_1‐1/p_0}}$ and $n^{‐\tilde \theta}$.

Let $q> 2$. Taking into account (\ref{mt_pr_mt_mthatline}) and (\ref{2hm0m0provrmo}), we get that $A$ is divided into
$$
A_1=\{(t, \, m):\; 0\le t\le \hat t(n), \; m\ge m'_t\},
$$
$$
A_2=\{(t, \, m):\; 0\le t\le \hat t(n), \; (\hat m_t)_+\le m\le m'_t\}
$$
(see Fig. 10).

\begin{figure}[h]
\vspace*{-3mm}
\begin{center} \resizebox{50mm}{!}
{\includegraphics{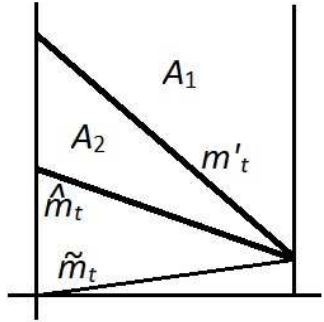}} \vspace*{-3mm}

\parbox[t]{60mm}{\caption{The partition of $A$.}\label{figure10}}
\end{center}
\vspace*{-4mm}
\end{figure}

We apply (\ref{s30}), (\ref{s31}), (\ref{s32}) and assertions 4, 5 of Theorem \ref{dn_inters}. We get that (\ref{dlwtm_q}) holds in $A_1$ (the progression strictly decreases with $m$), and in $A_2$, (\ref{dlwtm_p1}) holds for $p_1\ge 2$, (\ref{dlwtm_2}) holds for $p_1< 2$ (since $s_*+\frac{1}{p_0}‐\frac{1}{p_1}<0$, the progression strictly increases with $m$ for $p_1\le 2$). For $p_1\le 2$, it is sufficient to calculate $2^{((1‐\lambda)\mu_*‐\lambda\alpha_*)k_*t}\cdot 2^{‐m(1‐\lambda)s_*}$ in $(0, \, m'_0)$ and $(\hat t(n), \, \hat m_{\hat t(n)})$; for $p_1>2$, in addition, we calculate $2^{\mu_*k_*t}\cdot 2^{‐m(s_*+1/q‐1/p_1)}d_n(B_{p_1}^{\nu_{t,m}}, \, l_q^{\nu_{t,m}})$ in $(0, \, \hat m_0)$. Taking into account (\ref{mt_pr_eq}), (\ref{tilmt_hatmt_eq}) and (\ref{hat_sigma_m0pr}), we get $n^{‐\hat \sigma}$ and $n^{‐\tilde \theta}$, and for, $p_1>2$, in addition, we get $n^{‐s_*}$.

This together with (\ref{9tiltheta_est}) yields the desired estimates for the widths (see Notation \ref{not6}, case 1).

{\bf Case 2.} Let $\mu_*+\alpha_*+\gamma_*/p_0‐\gamma_*/p_1< 0$, $\mu_*+\alpha_*> 0$. Then $\tilde m_t\le 0$, $m_t\le 0$ for $t\ge 0$.

Let $q\le 2$. We define the number $\hat t(n)$ by $\hat m_{\hat t(n)}=0$. Then 
\begin{align}
\label{theta3lam000}
\sup _{f\in M}\|f\|_{Y_q(\tilde \Omega _{[\hat t(n)]})} \stackrel{(\ref{nu2}), (\ref{emb_nu})}{\underset{\mathfrak{Z}_0}{\lesssim}} 2^{((1‐\lambda)\mu_*‐\lambda \alpha_*)k_*\hat t(n)} \stackrel{(\ref{hat_mt})}{=} n^{((1‐\lambda)\mu_*‐\lambda \alpha_*)/\gamma_*} = n^{‐\theta _2}
\end{align} 
(see Notation \ref{not6}, case 2, subcase $q\le 2$).

By (\ref{s30}), (\ref{s31}) and assertion 6 of Theorem \ref{dn_inters}, we get (\ref{dlwtm_q}) for $(t, \, m)\in A$; the progression strictly decreases with $m$. Hence it is sufficient to calculate $2^{((1‐\lambda)\mu_*‐\lambda\alpha_*)k_*t}\cdot 2^{‐m(1‐\lambda)s_*}$ in $(0, \, \hat m_0)$ and $(\hat t(n), \, 0)$. We get $n^{‐\theta_1}$ and $n^{‐\theta_2}$ (see Notation \ref{not6}, case 2, subcase $q\le 2$). This together with (\ref{theta3lam000}) yields the estimate for the widths.

Let $q> 2$. We define the number $\hat t(n)$ by $m'_{\hat t(n)}=0$. By (\ref{nu2}), (\ref{emb_nu}), (\ref{mtpr0eq}), we get (\ref{9est_nu}).

We define the number $t_1(n)$ by $\hat m_{t_1(n)}=0$. From Proposition \ref{m_pr_t0_hat_mt} it follows that $t_1(n)<\hat t(n)$, $\overline{m}_{\hat t(n)}>0$.

Taking into account (\ref{2hm0m0provrmo}), we get that $A$ is divided into
$$
A_1=\{(t, \, m):\; 0\le t\le \hat t(n), \; m\ge m'_t\},
$$
$$
A_2=\{(t, \, m):\; 0\le t\le \hat t(n), \; (\hat m_t)_+\le m\le m'_t\}
$$
(see Fig. 11).
\begin{figure}[h]
\vspace*{-3mm}
\begin{center} \resizebox{50mm}{!}
{\includegraphics{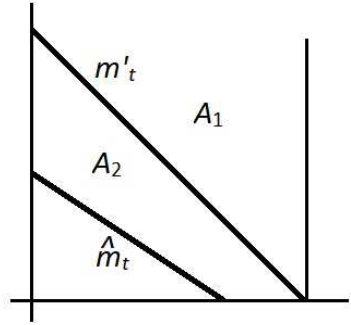}} \vspace*{-3mm}

\parbox[t]{60mm}{\caption{The partition of $A$.}\label{figure11}}
\end{center}
\vspace*{-4mm}
\end{figure}

By (\ref{s30})‐‐(\ref{s32}) and assertions 4, 5 of Theorem \ref{dn_inters}, (\ref{dlwtm_q}) holds in $A_1$ (the progression strictly decreases with $m$), and in $A_2$, (\ref{dlwtm_p1}) holds for $p_1\ge 2$, (\ref{dlwtm_2}) holds for $p_1< 2$ (if $p_1\le 2$, the progression strictly increases with $m$). Hence, for $p_1\le 2$, it is sufficient to calculate $2^{((1‐\lambda)\mu_*‐\lambda\alpha_*)k_*t}\cdot 2^{‐m(1‐\lambda)s_*}$ in $(0, \, m'_0)$ and $(\hat t(n), \, 0)$, and for $p_1>2$, in addition, $2^{\mu_*k_*t}\cdot 2^{‐m(s_*+1/q‐1/p_1)}d_n(B_{p_1}^{\nu_{t,m}}, \, l_q^{\nu_{t,m}})$ in $(0, \, \hat m_0)$ and $(t_1(n), \, 0)$. Taking into account (\ref{mt_pr_eq}), (\ref{mtpr0eq}) and (\ref{hat_sigma_m0pr}), we get $n^{‐\theta_j}$, $j=1, \, \dots, \, j_0$ (see Notation \ref{not6}, case 2, subcase $q>2$).

This together with (\ref{9est_nu}) yields the desired estimate for the widths.

{\bf Case 3.} Let $\mu_*+\alpha_*< 0$, $s_*+\frac 1q‐\frac{1}{p_1}<0$. Then $\mu_*+\alpha_*+\gamma_*/p_0‐\gamma_*/p_1< 0$, $\tilde m_t\le 0$, $m_t\ge 0$ for $t\ge 0$. 

Let $q\le 2$. We define the number $\hat t(n)$ by $\hat m_{\hat t(n)} = m_{\hat t(n)}$. By (\ref{nu4}), (\ref{emb_nu}), (\ref{hat_mt}), (\ref{mt_t}), we have (\ref{supfmthetahat}).

From (\ref{s30}), (\ref{s31}) and assertion 6 of Theorem \ref{dn_inters} it follows that (\ref{dlwtm_q}) holds in $A$; the progression strictly decreases with $m$. Hence it is sufficient to calculate $2^{((1‐\lambda)\mu_*‐\lambda\alpha_*)k_*t}\cdot 2^{‐m(1‐\lambda)s_*}$ in $(0, \, \hat m_0)$ and $(\hat t(n), \, \hat m_{\hat t(n)})$. Taking into account (\ref{mt_t}) and (\ref{mt_hatmt_eq}), we get $n^{‐\theta_1}$ and $n^{‐\theta_2}$ (see Notation \ref{not6}, case 3, subcase $q\le 2$). This together with (\ref{supfmthetahat}) yields the estimate for the widths.

Let $q> 2$. We define the number $\hat t(n)$ by $\overline{m}_{\hat t(n)} = m_{\hat t(n)}$. By (\ref{nu4}), (\ref{emb_nu}), (\ref{line_mt}), (\ref{mt_t}), we get (\ref{qth2sup}).

Notice that $m'_t>\hat m_t$ for $0\le t\le \hat t(n)$. Indeed, otherwise there exists $t^*\in [0, \, \hat t(n)]$ such that $\hat m_{t^*}=m'_{t^*}$ (it follows from (\ref{2hm0m0provrmo})). By (\ref{mt_pr_mt_mthatline}), we have $m'_{\hat t(n)} = m_{\hat t(n)}>0$ (this together with (\ref{2hm0m0provrmo}) implies that $m'_{t^*}>0$), $m'_{t^*}=\tilde m_{t^*}\le 0$. We arrive to a contradiction.

Taking into account (\ref{mt_pr_mt_mthatline}), we get that $A$ is divided into
$$
A_1=\{(t, \, m):\; 0\le t\le \hat t(n), \; m\ge m'_t\},
$$
$$
A_2=\{(t, \, m):\; 0\le t\le \hat t(n), \; (\hat m_t)_+\le m\le m'_t, \; m\ge m_t\},
$$
$$
A_3=\{(t, \, m):\; 0\le t\le \hat t(n), \; (\hat m_t)_+\le m\le m_t\}
$$
(see Fig. 12).

\begin{figure}[h]
\vspace*{-3mm}
\begin{center} \resizebox{50mm}{!}
{\includegraphics{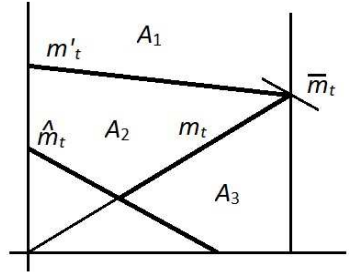}} \vspace*{-3mm}

\parbox[t]{60mm}{\caption{The partition of $A$.}\label{figure12}}
\end{center}
\vspace*{-4mm}
\end{figure}

We apply (\ref{s30})‐‐(\ref{s32}), (\ref{sved_k_p1}) and assertions 4--5 of Theorem \ref{dn_inters}. In $A_1$, (\ref{dlwtm_q}) holds (the progression strictly decreases with $m$), in $A_2$, we have (\ref{dlwtm_p1}) for $p_1\ge 2$ and (\ref{dlwtm_2}) for $p_1< 2$ (since $s_*+\frac 1q‐\frac{1}{p_1}<0$, in both cases, in the right‐hand sides of (\ref{dlwtm_p1}) and (\ref{dlwtm_2}) the progression strictly increases with $m$), and in $A_3$, we have (\ref{dlwtm_p1}) (the progression strictly increases with $m$). Hence it is sufficient to calculate $2^{((1‐\lambda)\mu_*‐\lambda\alpha_*)k_*t}\cdot 2^{‐m(1‐\lambda)s_*}$ in $(0, \, m'_0)$ and $(\hat t(n), \, \hat m_{\hat t(n)})$. Taking into account (\ref{mt_pr_eq}), (\ref{mt_linemt_eq}) and (\ref{hat_sigma_m0pr}), we get $n^{‐\theta_j}$, $j=1, \,  2$ (see Notation \ref{not6}, case 3, subcase $q>2$). This together with (\ref{qth2sup}) yields the desired estimates for the widths.

{\bf Case 4.} Let $\mu_*+\alpha_*< 0$, $s_*+\frac 1q‐\frac{1}{p_1}>0$. Then $\mu_*+\alpha_*+\gamma_*/p_0‐\gamma_*/p_1< 0$, $\tilde m_t\le 0$, $m_t\ge 0$ for $t\ge 0$.

First we consider the case $\gamma_*>0$.

Let $q\le 2$. We define the number $\hat t(n)$ by $\hat m_{\hat t(n)}=0$. By (\ref{nu1}), (\ref{emb_nu}), we have (\ref{supfmmugamma}).

The domain $A$ is divided into
$$
A_1=\{(t, \, m):\; 0\le t\le \hat t(n), \;  m\ge (\hat m_t)_+, \; m\ge m_t\},
$$
$$
A_2=\{(t, \, m):\; 0\le t\le \hat t(n), \; (\hat m_t)_+\le m\le m_t\}
$$
(see Fig. 13).

\begin{figure}[h]
\vspace*{-3mm}
\begin{center} \resizebox{50mm}{!}
{\includegraphics{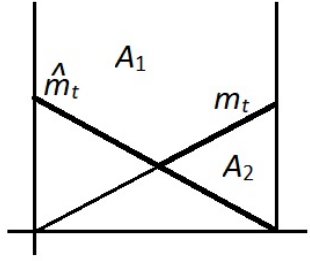}} \vspace*{-3mm}

\parbox[t]{60mm}{\caption{The partition of $A$.}\label{figure13}}
\end{center}
\vspace*{-4mm}
\end{figure}

By (\ref{s30})‐‐(\ref{s31}), (\ref{sved_k_p1}) and assertion 6 of Theorem \ref{dn_inters}, (\ref{dlwtm_q}) holds in $A_1$, (\ref{dlwtm_p1}) holds in $A_2$; the progressions strictly decrease with $m$. We define the number $t_1(n)$ by $\hat m_{t_1(n)}=m_{t_1(n)}$. It is sufficient to calculate $2^{((1‐\lambda)\mu_*‐\lambda\alpha_*)k_*t}\cdot 2^{‐m(1‐\lambda)s_*}$ in $(0, \, \hat m_0)$, and $2^{\mu_*k_*t}\cdot 2^{‐m(s_*+1/q‐1/p_1)}$ in $(t_1(n), \, \hat m_{t_1(n)})$ and $(\hat t(n), \, 0)$. Taking into account (\ref{hat_mt}) and (\ref{mt_hatmt_eq}), we get $n^{‐\theta_j}$, $1\le j\le 3$ (see Notation \ref{not6}, case 4, subcase $q\le 2$). This together with (\ref{supfmmugamma}) yields the estimate for the widths.

Let $q> 2$. We define $\hat t(n)$ by $\overline{m}_{\hat t(n)}=0$. We apply (\ref{nu1}), (\ref{emb_nu}) and obtain (\ref{sup_f_m_qmu}).

As in the previous case, we get that if $t\ge 0$, $m'_t\ge m_t$, then $m'_t>\hat m_t$.

Taking into account (\ref{mt_pr_mt_mthatline}), we get that $A$ is divided into
$$
A_1=\{(t, \, m):\; 0\le t\le \hat t(n), \; m\ge m'_t, \; m\ge m_t\},
$$
$$
A_2=\{(t, \, m):\; 0\le t\le \hat t(n), \; (\hat m_t)_+\le m\le m'_t, \; m\ge m_t\},
$$
$$
A_3=\{(t, \, m):\; 0\le t\le \hat t(n), \; (\hat m_t)_+\le m\le \overline{m}_t, \; m\le m_t\},
$$
$$
A_4=\{(t, \, m):\; 0\le t\le \hat t(n), \; m\le m_t, \; m\ge \overline{m}_t\}
$$
(see Fig. 14).

\begin{figure}[h]
\vspace*{-3mm}
\begin{center} \resizebox{50mm}{!}
{\includegraphics{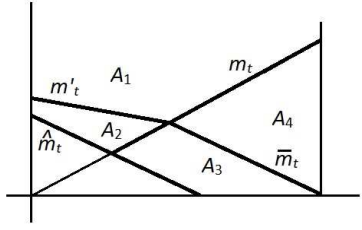}} \vspace*{-3mm}

\parbox[t]{60mm}{\caption{The partition of $A$.}\label{figure14}}
\end{center}
\vspace*{-4mm}
\end{figure}

By (\ref{s30})‐‐(\ref{s32}), (\ref{sved_k_p1}) and assertions 4, 5 of Theorem \ref{dn_inters}, (\ref{dlwtm_q}) holds in $A_1$ (the progression strictly decreases with $m$), in $A_2$,  (\ref{dlwtm_p1}) holds for $p_1\ge 2$, and  (\ref{dlwtm_2}) holds for $p_1< 2$ (if $p_1\le 2$, the progression strictly increases with $m$), in $A_3\cup A_4$, we get (\ref{dlwtm_p1}) (in $A_3$ the progression strictly increases for $p_1\le 2$, in $A_4$, the progression strictly decreases). Notice that, for $p_1\ge 2$, in $A_2\cup A_3$ the right‐hand side of (\ref{dlwtm_p1}) has the order 
$$2^{\mu_*k_*t}\cdot 2^{‐m(s_*+1/q‐1/p_1)}(n^{‐\frac 12}2^{\frac{\gamma_*k_*t}{q}}\cdot 2^{\frac mq})^{\frac{1/p_1‐1/q}{1/2‐1/q}}.$$ We define the numbers $t_1(n)$ and $t_2(n)$ by $\overline{m}_{t_1(n)} =m_{t_1(n)}$, $\hat m_{t_2(n)}=0$.
It is sufficient to calculate $2^{((1‐\lambda)\mu_*‐\lambda\alpha_*)k_*t}\cdot 2^{‐m(1‐\lambda)s_*}$ in $(0, \, m'_0)$ and $2^{\mu_*k_*t}\cdot 2^{‐m(s_*+1/q‐1/p_1)}$ in $(\hat t(n), \, 0)$ and $(t_1(n), \, \overline{m}_{t_1(n)})$, and for $p_1>2$, in addition, we calculate the value $2^{\mu_*k_*t}\cdot 2^{‐m(s_*+1/q‐1/p_1)}d_n(B_{p_1}^{\nu_{t,m}}, \, l_q^{\nu_{t,m}})$ in $(0, \, \hat m_0)$ and $(t_2(n), \, 0)$. Taking into account (\ref{hat_mt}), (\ref{mt_linemt_eq}) and (\ref{hat_sigma_m0pr}), we get $n^{‐\theta_j}$, $1\le j\le j_0$ (see Notation 6, case 4, subcase $q>2$). This together with (\ref{sup_f_m_qmu}) yields the estimate for the widths.

For $\gamma_*=0$, the proof is similar; the number $\hat t(n)$ is such that $2^{k_*\hat t(n)}$ is a sufficiently large degree of $n$.
\end{proof}

\begin{Rem}
\label{rem1} Let $\mu_*+\alpha_*+\frac{\gamma_*}{p_0} ‐\frac{\gamma_*}{p_1}\ge 0$ and one of the following conditions holds:
\begin{enumerate}
\item $p_0\ge q$, $p_1\ge q$;

\item $p_0>q>p_1$, $s_*+\frac{1}{p_0}‐\frac{1}{p_1}< 0$.
\end{enumerate}
Then the assertion of Theorem \ref{main} holds without the condition (\ref{pef1}).
\end{Rem}

The proof is similar; here we use (\ref{up_est0}). Since $p_0\ge q$ and $\mu_*+\alpha_*+\frac{\gamma_*}{p_0} ‐\frac{\gamma_*}{p_1}\ge 0$, we have (\ref{emb_nu}), where $\nu_*$ is defined by (\ref{nu3}); recall that it follows from (\ref{f_yqe}) and H\"{o}lder's inequality. We get that $\hat m_t >0$ holds for $t< \hat t(n)$; we set $s_t=\nu_{t,0}$ and obtain $d_{s_t}(B_{p_0}^{\nu_{t,0}}, \, l_q^{\nu_{t,0}})=0$.

\section{The estimates for the widths of the intersection of weighted Sobolev classes}

Recall the definitions of a John domain and of an $h$‐set.

We denote by $B_a(x)$ the Euclidean ball of radius $a$ centered at the point $x$.

\begin{Def}
\label{fca} Let $\Omega\subset\R^d$ be a bounded domain, $a>0$. We write $\Omega \in {\bf FC}(a)$ if there is a point $x_*=x_*(\Omega)\in \Omega$
such that for each $x\in \Omega$ there are a number $T(x)>0$ and a curve $\gamma _x:[0, \, T(x)] \rightarrow\Omega$ with the following properties:
\begin{enumerate}
\item $\gamma _x$ has the natural parametrization with respect to the Euclidean norm on $\R^d$,
\item $\gamma _x(0)=x$, $\gamma _x(T(x))=x_*$,
\item $B_{at}(\gamma _x(t))\subset \Omega$ for all $t\in [0, \, T(x)]$.
\end{enumerate}
We say that $\Omega$ is a John domain if $\Omega\in {\bf FC}(a)$ for some $a>0$.
\end{Def}

\begin{Def}
\label{h_set} {\rm (see \cite{m_bricchi1}).} Let $\Gamma\subset
\R^d$ be a non‐empty compact set, let $h:(0, \, 1] \rightarrow (0, \,
\infty)$ be a non‐decreasing function. We say that $\Gamma$ is an
$h$-set if there are $c_*\ge 1$ and a finite
$\sigma$‐additive measure $\mu$ on $\R^d$ such that $\supp
\mu=\Gamma$ and
$$
c_*^{-1}h(t)\le \mu(B_t(x))\le c_* h(t)
$$
for all $x\in \Gamma$ and $t\in (0, \, 1]$.
\end{Def}

Below ${\rm mes}$ is the Lebesgue measure on $\R^d$, $Y_q(\Omega) = L_{q,v}(\Omega)$, ${\cal P}(\Omega) = {\cal P}_{r‐1}(\Omega)$ is the space of polynomials of degree at most $r‐1$, $X_{p_0}(\Omega) = L_{p_0,w}(\Omega)$. When estimating the widths of $\widehat M$ defined by (\ref{widehat_m}), we set
\begin{align}
\label{xp1def}
\begin{array}{c}
X_{p_1}(\Omega)=\left\{f:\Omega \rightarrow \R: \; \left\|\frac{\nabla^r f}{g}\right\|^{p_1}_{L_{p_1}(\Omega)} + \left\|\frac{f}{g_0}\right\|^{p_1}_{L_{p_1}(\Omega)}<\infty\right\},
\\
\|f\|_{X_{p_1}(\Omega)}= \left(\left\|\frac{\nabla^r f}{g}\right\|^{p_1}_{L_{p_1}(\Omega)} + \left\|\frac{f}{g_0}\right\|^{p_1}_{L_{p_1}(\Omega)}\right)^{1/p_1}.
\end{array}
\end{align}
When estimating the widths of $M$ defined by (\ref{m_def}), we set
\begin{align}
\label{xp1def1}
\begin{array}{c}
X_{p_1}(\Omega)=\left\{f:\Omega \rightarrow \R: \; \left\|\frac{\nabla^r f}{g}\right\|_{L_{p_1}(\Omega)}<\infty\right\},
\\
\|f\|_{X_{p_1}(\Omega)}= \left\|\frac{\nabla^r f}{g}\right\|_{L_{p_1}(\Omega)}.
\end{array}
\end{align}

First we consider $\widehat M$ in $L_{q,v}(\Omega)$ defined by (\ref{widehat_m}), where $g$, $g_0$, $w$, $v$ are defined by (\ref{gwv}), (\ref{g0_def}), $\Omega \subset \left(‐\frac 12, \, \frac 12\right)^d$ is a John domain, $\Gamma \subset \partial \Omega$ is an $h$‐set, $h$ is defined by (\ref{h_theta}).

We show that for $\widehat{M}$ Assumptions \ref{supp1}‐‐\ref{supp6} hold with
\begin{align}
\label{1supp1_6}
\gamma_*=\theta, \;\; s_*=\frac rd, \;\; \mu_* = \beta + \lambda ‐ r ‐\frac dq + \frac{d}{p_1}, \;\; \alpha _* = \sigma ‐\lambda +\frac dq ‐\frac{d}{p_0}. 
\end{align}

We define the partitions $\{\Omega_{t,j}\}_{t\ge t_0, j\in \hat J_t}$ and $T_{t,j,m}$ as in \cite{vas_inters} (\S 4, proof of Theorem 1). Notice that there are numbers $b_*=b_*(a, \, d)>0$, $\overline{s}=\overline{s}(a, \, d)\in \N$ such that $\Omega _{t,j}\in {\bf FC}(b_*)$,
\begin{align}
\label{diam_otj} 
\begin{array}{c}
{\rm diam}\, \Omega _{t,j} \underset{a,d}{\asymp} 2^{‐\overline{s}t}, \quad g(x) \underset{\mathfrak{Z}_0}{\asymp} 2^{\beta\overline{s}t}, \; g_0(x)\underset{\mathfrak{Z}_0}{\asymp} 2^{(\beta‐r)\overline{s}t}, \\ w(x) \underset{\mathfrak{Z}_0}{\asymp} 2^{\sigma\overline{s}t},\; v(x) \underset{\mathfrak{Z}_0}{\asymp} 2^{\lambda\overline{s}t}, \; x\in \Omega_{t,j}.
\end{array}
\end{align}
If $E\in T_{t,j,m}$, then there is $b_{**}(a, \, d)>0$ such that $E\in {\bf FC}(b_{**}(a, \, d))$,
\begin{align}
\label{mes_e} {\rm mes}\, E \underset{a,d}{\lesssim} 2^{‐\overline{s}td‐m}, \quad {\rm diam}\, E \underset{a,d}{\lesssim} 2^{‐\overline{s}t‐\frac md}.
\end{align}

Assumption \ref{supp1} follows from (\ref{diam_otj}), the embedding theorem \cite{resh1, resh2} and H\"{o}lder's inequality.

Assumptions \ref{supp2}‐‐\ref{supp4}, \ref{supp6} can be proved as in \cite{vas_inters} (see \S 4--5, proofs of the upper and the lower estimates for Theorem 1).

Let us prove that Assumption \ref{supp5} holds. 

We show that for $E\in T_{t,j,m}$
\begin{align}
\label{mes_e_low} {\rm mes}\, E \underset{a,d}{\gtrsim} 2^{‐\overline{s}td‐m}, \quad {\rm diam}\, E \underset{a,d}{\gtrsim} 2^{‐\overline{s}t‐\frac md}.
\end{align}
To this end, we recall how to construct the partitions $T_{t,j,m}$ according to \cite{vas_john}. 

Let $t\ge t_0$, $j\in \hat J_t$ be fixed; we set $G=\Omega_{t,j}$. Let $\Theta(G)$ be the Whitney's covering of $G$. Each element of $\Theta(G)$ is a dyadic cube; if the intersection of $\Delta$, $\Delta'\in \Theta(G)$ has dimension $d‐1$, then
\begin{align}
\label{mes_d_d1}
\frac{{\rm mes}(\Delta)}{{\rm mes}(\Delta')} \underset{d}{\asymp} 1. 
\end{align}
In \cite[Lemma 3]{vas_john}, a tree ${\cal T}$ with the root $\omega_*$ was constructed, as well as the bijection $F:{\bf V}({\cal T})\rightarrow \Theta(G)$ (here ${\bf V}({\cal T})$ is the set of all vertices of ${\cal T}$) with the following properties:
\begin{enumerate}
\item If the vertices $\omega$ and $\omega'$ are adjacent, then $F(\omega)$ and $F(\omega')$ have the $d‐1$‐dimensional intersection.
\item There are numbers $c_1=c_1(a, \, d)\in \N$ and $c_2=c_2(a, \, d)\in \N$ such that for all vertices $\omega>\omega'$
\begin{align}
\label{momega} \rho(\omega, \, \omega')\le c_1(m_\omega‐m_{\omega'})+c_2,
\end{align} 
where $\rho(\omega, \, \omega')$ is the distance between $\omega$ and $\omega'$, $2^{‐m_\omega}$ and $2^{‐m_{\omega'}}$ is the length of the edge of $F(\omega)$ and $F(\omega')$, respectively.
\end{enumerate}

Let ${\cal T}'$ be a subtree of ${\cal T}$ with the minimal vertex $\omega'$, let ${\bf V}({\cal T}')$ be its vertex set, and let $\tilde G_{{\cal T}'} = \cup _{\omega\in {\bf V}({\cal T}')} F(\omega)$. In \cite{vas_john} (see formula (3.1) and Corollary 1) the domain 
\begin{align}
\label{g_cal_t_pr}
G_{{\cal T}'}\in {\bf FC}(b_{**}(a, \, d)), \quad x_*(G_{{\cal T}'}) \in F(\omega'),
\end{align}
was constructed; $G_{{\cal T}'}\bigtriangleup \tilde G_{{\cal T}'}$ has the zero measure, the point $x_*(G_{{\cal T}'})$ is from Definition \ref{fca}.

Let $\omega$ be a vertex of ${\cal T}$. We denote by ${\cal T}_\omega$ the tree with vertex set $\{\omega':\; \omega'\ge \omega\}$, and by ${\bf V}_1(\omega)$, the set of vertices that follow the vertex $\omega$: $${\bf V}_1(\omega)=\{\omega'>\omega:\; \omega'\text{ is adjacent with }\omega\}.$$

From (\ref{mes_d_d1}) and property 1 of $F$ we get that there is a number $k(d)\in \N$ such that ${\rm card}\, {\bf V}_1(\omega)\le k(d)$ for each vertex $\omega$.

The partition $T_{t,j,m}$ was constructed in two steps.

\begin{enumerate}
\item First we construct the partition $T'_{t,j,m}$ (see Lemma 5 from \cite{vas_john}). For each vertex $\omega$ we construct the partition $P_\omega$ of $G_{{\cal T}_\omega}$ (see Lemma 4 from \cite{vas_john}). If $${\rm mes}(G_{{\cal T}_\omega})\le (k(d)+1){\rm mes}(G)\cdot 2^{‐m},$$ we set $P_\omega=\{G_{{\cal T}_\omega}\}$. If 
\begin{align}
\label{mes_g_kd1}
{\rm mes}(G_{{\cal T}_\omega})> (k(d)+1){\rm mes}(G)\cdot 2^{‐m},
\end{align}
we find a vertex $\hat \omega \ge \omega$ such that
\begin{align}
\label{mes_g_t_g_mg}
{\rm mes}(G_{{\cal T}_{\hat \omega}}) > {\rm mes}(G_{{\cal T}_\omega})‐ 2^{‐m}{\rm mes}(G) \stackrel{(\ref{mes_g_kd1})}{\ge} k(d)\cdot 2^{‐m}{\rm mes}(G),
\end{align}
and for each $\omega'\in {\bf V}_1(\hat \omega)$,
$$
{\rm mes}(G_{{\cal T}_{\omega'}}) \le {\rm mes}(G_{{\cal T}_\omega})‐ 2^{‐m}{\rm mes}(G).
$$
By (\ref{g_cal_t_pr}), (\ref{mes_g_kd1}) and (\ref{mes_g_t_g_mg}), ${\rm mes}(F(\omega)) \underset{a,d}{\gtrsim} 2^{‐m}{\rm mes}(G)$, 
\begin{align}
\label{mes_l_e0}
{\rm mes}(F(\hat\omega)) \underset{a,d}{\gtrsim} 2^{‐m}{\rm mes}(G); 
\end{align}
from (\ref{mes_d_d1}) and property 1 of $F$ it follows that if $\omega'\in {\bf V}_1(\hat \omega)$, then ${\rm mes}(F(\omega')) \underset{a,d}{\gtrsim} 2^{‐m}{\rm mes}(G)$. Hence
\begin{align}
\label{mes_l_e1} {\rm mes}\, (G_{{\cal T}_{\omega'}})\underset{a,d}{\gtrsim} 2^{‐m}{\rm mes}(G), \; \omega'\in {\bf V}_1(\hat \omega),
\end{align}
and if $\hat\omega>\omega$, then
\begin{align}
\label{mes_l_e2} {\rm mes}(G_{{\cal T}_\omega \backslash {\cal T}_{\hat \omega}})\underset{a,d}{\gtrsim} 2^{‐m}{\rm mes}(G).
\end{align}
The partition $P_\omega$ consists of $G_{{\cal T}_\omega \backslash {\cal T}_{\hat \omega}}$ (if $\hat\omega>\omega$), $F(\hat \omega)$ and $G_{{\cal T}_{\omega'}}$, $\omega'\in {\bf V}_1(\hat \omega)$.

Now the partition $T'_{t,j,m}$ is defined as follows. First we construct $P_{\omega_*}$. If $P_{\omega_*}\ne \{G_{{\cal T}_{\omega_*}}\}$, then for each $\omega'\in {\bf V}_1(\hat \omega)$ we construct $P_{\omega'}$, and so on. When we stop dividing the subtrees, we get the desired partition $T'_{t,j,m}$. From (\ref{mes_l_e0})‐‐(\ref{mes_l_e2}) it follows that
\begin{align}
\label{mes_l_e3} {\rm mes}(E) \underset{a,d}{\gtrsim} 2^{‐m}{\rm mes}(G), \; E\in T'_{t,j,m}.
\end{align}
\item Now we construct $T_{t,j,m}$; it is a subdivision of $T'_{t,j,m}$. If $F(\omega)\in T'_{t,j,m}$, ${\rm mes}(F(\omega)) > 2^{‐m}{\rm mes}(G)$, then we take the uniform division of $F(\omega)$ into dyadic cubes $\Delta$, 
\begin{align}
\label{mes_l_e4}
2^{‐m}{\rm mes}(G)\underset{d}{\lesssim}{\rm mes}(\Delta) \le 2^{‐m}{\rm mes}(G).
\end{align}
\end{enumerate}
From (\ref{mes_l_e3})‐‐(\ref{mes_l_e4}) we get (\ref{mes_e_low}). This together with (\ref{mes_e}) implies that
\begin{align}
\label{mes_l_e5}
{\rm mes}\, E \underset{a,d}{\asymp} 2^{‐\overline{s}td‐m}, \quad {\rm diam}\, E \underset{a,d}{\asymp} 2^{‐\overline{s}t‐\frac md}, \quad E\in T_{t,j,m}.
\end{align}

Now we prove that Assumption \ref{supp5} holds.

The inequality (\ref{pef}) is checked as in \cite{vas_inters} (see \S 4, proof of the upper estimate for Theorem 1). The estimate (\ref{pef1}) can be proved similarly as (\ref{pef}). Here we use the following fact: if $B$ is a ball of radius $\rho$, $f\in W^r_{p_1}(B)$, then by the H\"{o}lder's inequality
$$
\rho ^{‐d+\frac{d}{q}}\|f\|_{L_1(B)} \underset{p_1,r,d}{\lesssim} \rho ^{r +\frac dq ‐\frac{d}{p_1}} \left(\|\nabla^r f\| _{L_{p_1}(B)} + \rho ^{‐r}\|f\|_{L_{p_1}(B)}\right).
$$

Let us prove (\ref{ptmf_to_f}). Recall how the operator $P_E$ is defined. Let $x_*(E)\in E$ be the point from Definition \ref{fca}, let $B_E\subset E$ be the ball centered at $x_*(E)$, ${\rm diam}\, B_E \underset{a,d}{\asymp} {\rm diam}\, E$. We take the orthogonal projection from  $L_2(B_E)$ onto ${\cal P}_{r‐1}(B_E)$, then extend it onto $L_1(B_E)$ as a continuous operator. After that we extend each polynomial onto $\R^d$ and multiply it by $\chi_E$. Since $L_{q,v}(B_E)\subset L_1(B_E)$, the operator $P_E$ is well‐defined on $L_{q,v}(\Omega)$.

If $f\in L_{q,v}(\Omega)$, then
$$
\|P_E(f\cdot \chi_E)\|_{L_q(E)} \underset{a,d,q}{\lesssim} \|P_E(f\cdot \chi_E)\|_{L_q(B_E)} \underset{d,q}{\lesssim}$$$$\lesssim ({\rm mes}\, B_E)^{1/q‐1}\|P_E(f\cdot \chi_E)\|_{L_1(B_E)} \underset{d,q}{\lesssim} ({\rm mes}\, B_E)^{1/q‐1}\|f\|_{L_1(B_E)}\le \|f\|_{L_q(E)}.
$$
This together with (\ref{diam_otj}) implies that
\begin{align}
\label{bound_oper} \left\| \sum \limits _{E\in T_{t,j,m}}P_E(f\cdot \chi_E)\right\| _{L_{q,v}(\Omega_{t,j})} \underset{a,d,q}{\lesssim} \|f\|_{L_{q,v}(\Omega_{t,j})}.
\end{align}
If $f\in C_0^\infty(\Omega_{t,j})$, then (\ref{ptmf_to_f}) follows from (\ref{diam_otj}), (\ref{mes_e}) and the estimate 
$$
\|f‐P_E(f\cdot \chi_E)\|_{L_q(E)} \underset{q,d,a}{\lesssim} (2^{‐m}{\rm mes}\, \Omega_{t,j})^{1/d} \|\nabla f\|_{L_q(E)}, \quad E\in T_{t,j,m}
$$
(see \cite[Lemma 8]{vas_width_raspr}).
The space $C_0^\infty(\Omega_{t,j})$ is dense in $L_{q,v}(\Omega_{t,j})$. This together with (\ref{bound_oper}) yields that (\ref{ptmf_to_f}) holds for each function $f\in L_{q,v}(\Omega_{t,j})$.

Let us prove (\ref{fpef}). Let $E\in T_{t,j,m}$, $E'\in T_{t,j,m\pm 1}$, ${\rm mes}(E\cap E')>0$. First we define the domain $G_{E,E'} \subset E\cup E'$. According to the construction of $T_{t,j,m}$, we have the following cases.

\begin{enumerate}
\item Let $E=G_{{\cal A}}$, $E'=G_{{\cal A}'}$, where ${\cal A}$ and ${\cal A}'$ are subtrees of ${\cal T}$ with the minimal vertices $\omega$ and $\omega'$. Then $\omega$ are $\omega'$ comparable; hence they can be joint by a chain ${\cal S}$. Let $G_{E,E'}= G_{{\cal S}}$. Notice that ${\rm mes}\, F(\omega) \underset{a,d}{\asymp} {\rm mes}(E)\stackrel{(\ref{mes_l_e5})}{\underset{a,d}{\asymp}} {\rm mes}(E') \underset{a,d}{\asymp} {\rm mes}\, F(\omega')$. By (\ref{momega}), ${\cal S}$ has at most $C(a, \, d)$ vertices. Therefore $G_{E,E'}\in {\bf FC}(b_1(a,\, d))$, where $b_1(a,\, d)>0$; we can take as $x_*(G_{E,E'})$ from Definition \ref{fca} the center of $F(\omega)$, as well as the center of $F(\omega')$.

\item Let $E'\subset F(\omega')$ be a dyadic cube, $E=G_{{\cal A}}$, where ${\cal A}$ is a subtree of ${\cal T}$ with the minimal vertex $\omega$ (the case when $E\subset F(\omega)$, $E'=G_{{\cal A}'}$ is considered similarly). Then $\omega'$ is a vertex of ${\cal A}$. As in the previous case we join $\omega$ and $\omega'$ by a chain ${\cal S}$ and set $G_{E,E'}= G_{{\cal S}}$. Again $G_{E,E'}\in {\bf FC}(b_1(a,\, d))$, ${\rm mes}\, F(\omega) \stackrel{(\ref{mes_l_e5})}{\underset{a,d}{\asymp}} {\rm mes}\, E'$; hence we can take as the point $x_*(G_{E,E'})$ from Definition \ref{fca} the center of $F(\omega)$, as well as the center of $E'$.

\item Let $E$ and $E'$ be dyadic cubes, $E\subset E'$ or $E'\subset E$. By (\ref{mes_l_e5}), ${\rm mes}\, E\underset{a, \, d}{\asymp} {\rm mes}\, E'$. We set $G_{E,E'}=E\cup E'$. Then $G_{E,E'}\in {\bf FC}(b_2(a, \, d))$; we can take as the point $x_*(G_{E,E'})$ from Definition \ref{fca} the center of $E$, as well as the center of $E'$.
\end{enumerate}

In all cases
\begin{align}
\label{diamgee_diamee} {\rm diam}\, G_{E,E'} \underset{a,d}{\asymp} {\rm diam}\,(E\cup E').
\end{align}

Let $B_{E,E'}\supset E\cup E'$ be a ball, 
\begin{align}
\label{diam_bee_diam_ee}
{\rm diam}\, B_{E,E'} \underset{a,d}{\asymp} {\rm diam}(E\cup E'); 
\end{align}
we extend the polynomials $\tilde P_Ef$ and $\tilde P_{E'}f$ onto $\R^d$ (the notation will be the same).

Applying the embedding theorem, we get
$$
\|\tilde P_Ef‐\tilde P_{E'}f\| _{L_q(E\cup E')} \le \|\tilde P_Ef‐\tilde P_{E'}f\| _{L_q(B_{E,E'})} \underset{d,q}{\lesssim} 
$$
$$
\lesssim ({\rm mes}\, B_{E,E'})^{1/q‐1}\|\tilde P_Ef‐\tilde P_{E'}f\| _{L_1(B_{E,E'})}\stackrel{(\ref{diamgee_diamee}), (\ref{diam_bee_diam_ee})}{\underset{a,d,q}{\lesssim}} 
$$
$$
\lesssim({\rm mes}\, B_{E,E'})^{1/q‐1}\|\tilde P_Ef‐\tilde P_{E'}f\| _{L_1(G_{E,E'})}\le
$$
$$
\le ({\rm mes}\, B_{E,E'})^{1/q‐1}(\|f‐\tilde P_{E'}f\| _{L_1(G_{E,E'})}+ \|f‐\tilde P_Ef\| _{L_1(G_{E,E'})})\underset{a,d,p_1,r}{\lesssim}
$$
$$
\lesssim ({\rm mes}\, B_{E,E'})^{1/q‐1}({\rm mes}\, G_{E,E'})^{\frac rd+1‐\frac{1}{p_1}} \|\nabla^r f\|_{L_{p_1}(E\cup E')}\stackrel{(\ref{mes_l_e5}), (\ref{diamgee_diamee}), (\ref{diam_bee_diam_ee})}{\underset{a,d,p_1,q,r}{\lesssim}}
$$
$$
\lesssim 2^{‐(\overline{s}td+m)\left(\frac rd+\frac 1q‐\frac {1}{p_1}\right)}\|\nabla^r f\|_{L_{p_1}(E\cup E')}.
$$

This together with (\ref{diam_otj}) yields (\ref{fpef}).

This completes the proof of  Assumption \ref{supp5}. Notice that in (\ref{fpef}) the space $X_{p_1}(\Omega)$ can be defined by (\ref{xp1def}), as well as by (\ref{xp1def1}). We obtain

\begin{Trm}
Let $\widehat M$ be defined by \eqref{widehat_m}, let $g$, $g_0$, $w$, $v$ be defined by \eqref{gwv}, \eqref{g0_def}, let $\Omega \subset \left(‐\frac 12, \, \frac 12\right)^d$ be a John domain, let $\Gamma \subset \partial \Omega$ be an $h$‐set, where $h$ is defined by \eqref{h_theta}. Let $s_*$, $\gamma_*$, $\mu_*$, $\alpha_*$ be defined according to \eqref{1supp1_6}. Then for $M:=\widehat{M}$ Theorem \ref{main} holds.

If one of the conditions from Remark \ref{rem1} holds, the same estimates are true for the set $M$ defined by (\ref{m_def}).
\end{Trm}

\vskip 0.3cm

Now we consider Example 2 from \cite{vas_inters}. Let $\Omega \subset \left(‐\frac 12, \, \frac 12\right)^d$ be a John domain, let $\Gamma \subset \partial \Omega$ be an $h$‐set, and let $\widehat{M}$ be defined by (\ref{widehat_m}), where
$$
g(x)=\varphi_g({\rm dist}\, (x, \, \Gamma)), \quad g_0(x)=\varphi_{g_0}({\rm dist}\, (x, \, \Gamma)), $$$$
w(x) =\varphi_w({\rm dist}\, (x, \, \Gamma)), \quad v(x)
=\varphi_v({\rm dist}\, (x, \, \Gamma)),
$$
$$ 
h(t) = (\log t)_*^{-\gamma}, \quad \gamma \ge 0,
$$
$$\varphi_g(t) = t^{-\beta} (\log t)_*^{\mu}, \quad
\varphi_{g_0}(t) = t^{r-\beta} (\log t)_*^{\mu}, \quad
\varphi_w(t) = t^{-\sigma} (\log t)_*^{\alpha}, \quad \varphi_v(t) =
t^{-\lambda} (\log t)_*^{\nu},
$$
$(\log t)_*=\max\{‐\log t, \, 1\}$; here we assume that
$$
\beta + \lambda = r +\frac dq -\frac{d}{p_1},
\quad \sigma -\lambda = \frac{d}{p_0} - \frac{d}{q}.
$$

As for the first example, we can check Assumptions \ref{supp1}‐‐\ref{supp6} with $s_*=\frac rd$, $\gamma_*=\gamma+1$, $\alpha_*=\alpha‐\nu$, $\mu_*=\mu+\nu$ (see also \cite{vas_inters}, proof of Theorem 2 in \S 4‐‐5). Hence Theorem \ref{main} holds with such $s_*$, $\gamma_*$, $\alpha_*$, $\mu_*$, $M:=\widehat{M}$. If one of the conditions from Remark \ref{rem1} holds, the same estimates are true for the set $M$ defined by (\ref{m_def}).

\vskip 0.3cm

In both examples there is another case when the widths of $\widehat{M}$ have the same orders as the widths of $M$. Let $\mu_*+\left(\frac{\gamma_*}{q}‐\frac{\gamma_*}{p_1}\right)_+<0$, $\mu_*+\alpha_*\le 0$, $\mu_*+\alpha_* + \frac{\gamma_*}{p_0} ‐\frac{\gamma_*}{p_1}\le 0$, $s_*+\frac{1}{\max\{p_0, \, q\}} ‐\frac{1}{p_1}>0$, and let $\lambda<\frac{d‐\theta}{q}$ (in the second example, $\theta:=0$). Then for the set $M$ the estimate from Theorem \ref{main} holds (here $j_0$ and $\theta_j$ are as in Notation \ref{not2}). It follows from the inclusions $\widehat{M} \subset M \subset W^r_{p_1, \, g}(\Omega)$, where
$$
W^r_{p_1, \, g}(\Omega) = \left\{f:\Omega \rightarrow \R:\; \left\| \frac{\nabla^r f}{g}\right\|_{L_{p_1}(\Omega)}\le 1\right\}.
$$

Indeed, we have
$$
d_n(\widehat{M}, \, L_{q,v}(\Omega)) \le d_n(M, \, L_{q,v}(\Omega)) \le d_n(W^r_{p_1, \, g}(\Omega), \, L_{q,v}(\Omega)).
$$
The left‐hand side can be estimated from below according to Theorem \ref{main}. The right‐hand side can be estimated from above by \cite[Theorem 1, cases 1, 3, 4]{vas_width_raspr}.

\vskip 0.3cm

Now we consider Example 3 from \cite{vas_inters}. The set $\widehat{M}$ is defined by (\ref{widehat_m}), where
$$
g(x) = (1 + |x|)^\beta, \quad g_0(x) = (1+|x|)^{\beta+r},
$$
$$
w(x) = (1+|x|)^\sigma, \quad v(x)=(1+|x|)^\lambda.
$$
As in the previous cases (see also \cite{vas_inters}, \S 4, 5, proof of Theorem 3), we check Assumptions \ref{supp1}‐‐\ref{supp6} with $s_*=\frac rd$, $\gamma_*=0$, $\mu_* = \beta+\lambda + r+\frac dq‐\frac{d}{p_1}$, $\alpha_*=\sigma‐ \lambda ‐\frac dq+\frac{d}{p_0}$. Hence the Kolmogorov widths of $\widehat{M}$ can be estimated according to Theorem \ref{main}; if one of the conditions of Remark \ref{rem1} holds, the same estimates are true for the set $M$ defined by (\ref{m_def}).

\end{document}